\documentclass[14pt]{amsart}
\usepackage{latexsym, amssymb}
\usepackage[usenames]{color}
\usepackage[all]{xy}

\newcommand{\be}{\begin{equation}}
\newcommand{\ee}{\end{equation}}
\newcommand{\beq}{\begin{eqnarray}}
\newcommand{\eeq}{\end{eqnarray}}

\newtheorem{thm}{Theorem} [section]
\newtheorem{cor}[thm]{Corollary}
\newtheorem{lem}[thm]{Lemma}

\newtheorem{prop}[thm]{Proposition}
\theoremstyle{definition}

\theoremstyle{remark}

\numberwithin{equation}{section}

\begin{document}
\title[$L$-functions of generalized Kloosterman sums and differential equations]
{ $L$-functions for families of generalized Kloosterman sums and $p$-adic differential equations}
\begin{abstract}


In this paper, we focus on a family of generalized Kloosterman sums over the torus.
With a few changes to Haessig and Sperber's
construction, we derive some relative $p$-adic cohomologies corresponding to the $L$-functions.
We present explicit forms of bases of top dimensional cohomology spaces, so to obtain a concrete method to compute lower bounds of Newton polygons of the $L$-functions. Using the theory of GKZ system,
we derive the Dwork's deformation equation for our family.
Furthermore, with the help of Dwork's dual theory and deformation theory, the strong Frobenius structure of this equation is established.
Our work adds some new evidences for Dwork's conjecture.
\end{abstract}

\author[C.L. Wang]{Chunlin Wang}
\address{School of Mathematical Sciences, Sichuan Normal University, Chengdu 610064, P.R. China}
\email{c-l.wang@outlook.com}

\author[L.P. Yang]{Liping Yang$^*$}
\address{School of Mathematical Sciences, Capital Normal University, Beijing 100048, P.R. China}
\email{yanglp2013@126.com}

\thanks{$^*$ L.P. Yang is the corresponding author and supported by Beijing Postdoctoral Research Foundation (No. 20540060001).
C.L. Wang is supported by National Natural Science Foundation of China (No. 11901415) and the China Scholarship Council (No. 201808515110).
}
\maketitle

\section{Introduction}

Let $p$ be a prime, and $\mathbb{F}_q$ be the finite field of $q$ elements with characteristic $p$.
For Laurent polynomial $g(x)\in \mathbb{F}_q[x_1^{\pm},\cdots ,x_n^{\pm}]$, the toric exponential sum is defined as
$$S_k(g):=\sum_{x\in (\mathbb{F}_{q^{k}}^{*})^{n}} \exp\Big(\frac{2\pi i}{p}Tr_k g(x)\Big),$$
where $Tr_k: \mathbb{F}_{q^k}\rightarrow \mathbb{F}_p$ and $\mathbb{F}_{q^{k}}^{*}$ denotes the set of non-zero elements in $\mathbb{F}_{q^{k}}$.
By a theorem of Dwork-Bombieri-Grothendieck, the following generating $L$-function is a rational function
$$L(g,T):=\exp\Big(\sum_{k=1}^{\infty} \frac{S_k(g)T^k}{k}\Big)=\frac{\prod_{i=1}^{d_1}(1-\alpha_i T)}{\prod_{j=1}^{d_2}(1-\beta_j T)}, $$
where $\alpha_{i}(1\le i\le d_1)$ and $\beta_j(1\le j\le d_2)$ are non-zero algebraic integers.
From Deligne's integrality theorem, one has the following estimates
$$|\alpha_i|_p=q^{-r_i}, |\beta_j|_p=q^{-s_j},r_i\in \mathbb{Q}\cap[0,n], s_j\in \mathbb{Q}\cap[0,n]$$
with normalized $p$-adic absolute value $|\ |_p$ such that $|q|_p=q^{-1}$.
The rational number $r_i$ (resp. $s_j$)
is called {\it slope} of $\alpha_i$ (resp. $\beta_j$) with respect to $q$.
The $p$-adic Riemann hypothesis for the $L$-function $L(g,T)$ is to determine the slopes of the zeros and poles.
This is an extremely hard question if there is no any smoothness condition on $g$.
When $g$ is nondegenerate, Adolphson and Sperber \cite{AS89} showed that
the $L$-function $L(g,T)^{(-1)^{n-1}}$ is a polynomial whose degree $m$ is determined by the volume of the Newton polyhedron of $g$.
If we write
$$L(g,T)^{(-1)^{n-1}}=\sum_{i=0}^mA_iT^i,$$
then the $q$-adic {\it Newton polygon} of $L(g,T)^{(-1)^{n-1}}$, denoted by $NP(g)$, is the lower convex closure in $\mathbb{R}^2$
of the following points
$$(r,ord_q A_r), r=0,1,...,m.$$
Further more, the $q$-adic Newton polygon of $L(g,T)^{(-1)^{n-1}}$ determines the slopes of its reciprocal roots.

The $L$-functions of classical Kloosterman families have been studied by Dwork \cite{[Dwo74]} and Sperber \cite{Sp1},\cite{Sp2} using
Dwork's $p$-adic theory.
For the Kloosterman family
$g^{(1)}(\Lambda,x)=x+\Lambda/x$ with parameter $\Lambda\in \mathbb{F}_q$,
Dwork derived that $\{0,1\}$ is the slope sequence of the $L$-function $L(g^{(1)}(\Lambda,x), T)$.
Later, Sperber \cite{Sp1},\cite{Sp2} generalized Dwork's results to the $n$ variables case given by
$$ g^{(n)}(\Lambda,x)=x_1+\cdots+x_n+\frac{\Lambda}{x_1\cdots x_n}$$
and proved that the slope sequence of $L(g^{(n)}(\Lambda,x),T)^{(-1)^{n-1}}$
should be $\{0,1,...,n\}$.

The classical Kloosterman family $g^{(1)}(\Lambda,x)$ was also studied by Robba \cite{Ro} using $p$-adic cohomology. Robba \cite{Ro} proved that the symmetric power $L$-function of this family is a polynomial with coefficients in $\mathbb{Z}$ and can be factored into two parts:
the trivial one and the one satisfying the functional equation. Fu and Wan \cite{[FW05]}, \cite{[FW08]}, \cite{[FW082]}, \cite{[FW10]} studied the hyper-Kloosterman family and its symmetric power $L$-functions using $l$-adic method.
Recently, Haessig and Sperber \cite{HS17} studied the following generalized Kloosterman family
$$\bar{F}(\Lambda,x):=\bar{f}(x)+\Lambda x^{\mu}\in \mathbb{F}_q[\Lambda][x_1^{\pm},\cdots,x_n^{\pm }],$$
 where $\bar{f}(x)$ is an arbitrary quasi-homogeneous nondegenerate Laurent polynomial and the deforming monomial
 $\Lambda x^{\mu}$ is an arbitrary monomial where $\mu$ does not lie on the affine hyperplane spanned by the support of $\bar{f}$.
Haessig and Sperber constructed some relative $p$-adic cohomologies to give sufficiently sharp estimates for the degree and the total degree of 
symmetric power $L$-functions.

In this paper, we focus on the following kind of generalized Kloosterman family defined by Haessig and Sperber
$$\bar{F}(\Lambda,x)=x_1^a+x_2^b+\frac{\Lambda}{x_1^cx_2^d},$$
 where $a,b,c,d$ are positive integers such that
$$\gcd(a,b)=\gcd(a,c)=\gcd(b,c)=\gcd(b,d)=1 \ {\rm and}\ p\nmid abcd.$$
Let $\bar{\mathbb{F}}_q$ be an algebraic closure of $\mathbb{F}_q$. Fix $\bar{\lambda}\in \bar{\mathbb{F}}_q^{*}$
and let $\lambda$ be its Teichm$\ddot{\rm {u}}$ller representative.
Let $\mathcal{C}_0$ denote the $p$-adic Banach space consisting of formal power series.
Let $F(\Lambda,x)$ be the Teichm${\rm \ddot{u}}$ller lifting of
$\bar{F}(\Lambda,x)$. Then let
$$F^{(0)}(\Lambda,x):=\pi F(\Lambda,x)$$
with $\pi^{p-1}=-p$.
For $l=1,2$, we define differential operator $D_{l,\Lambda}$ by
$$D_{l,\Lambda}:=x_l\frac{\partial}{\partial x_l}+x_l \frac{\partial F^{(0)}}{\partial x_l}.$$
Then Haessig and Sperber \cite{HS17} constructed the $p$-adic complex $\Omega^{\bullet}(\mathcal{C}_0,\nabla(D^{(\Lambda)}))$
such that
$$L(\bar{F}(\bar{\lambda},x), T)^{-1}=\det (1-Frob_{\lambda}T|H^2(\mathcal{C}_{0,\lambda}, \nabla (D^{(\lambda)})),$$
where $\mathcal{C}_{0,\lambda}$ and $D_{l,\lambda}$ are specializing of $\mathcal{C}_0$ and $D_{l,\Lambda}$
at $\Lambda=\lambda$, respectively.

One aim of this paper is to study the slopes of zeros of $L(\bar{F}(\bar{\lambda},x),T)^{-1}$.
For this purpose, we use relative $p$-adic cohomologies constructed by Haessig and Sperber \cite{HS17} but with a few modifications
such as using another Dwork's splitting function and letting $p>2$.
As we all know, the Newton polygon has a topological lower bound, called Hodge polygon.
The general way to calculate Newton polygon is to first compute its lower bound,
and then determine when the Newton polygon and the Hodge polygon coincide.
It is known from the work of Adolphson and Sperber \cite{AS89} that the Hodge polygon is totally determined by the weight of a basis of the top cohomology space $H^2(\Omega^{\bullet}(\mathcal{C}_0,\nabla(D^{(\Lambda)})))$, this makes us to study the explicit form of the basis. Our main tool is the specific reduction of elements of
a polynomial ring by the ideals generated by differential operators $\{D_{l,\lambda}\}_{l=1,2}$. To describe our results, we define
set $B:=\{ x_1^{v_1}x_2^{v_2}: v_1,v_2\in \mathbb{Z}\}$, where $(v_1, v_2)$ such that
\begin{align}\label{eq02}
&\left\{
  \begin{array}{ll}
  -c<v_1\le a,-d<v_2\le b, \frac{d-1}{c-1}(v_1-a)\le v_2< \frac{d-1}{c-1}v_1+b,&  \hbox{if $c>1$, $d>1$};\\
   -c<v_1\le a,-d<v_2\le b, (v_1,v_2)\neq (a,v_2)\ {\rm with}\ v_2=0,-1,...,1-d,& \hbox{if $c=1$, $d>1$};\\
-c<v_1\le a,-d<v_2\le b, (v_1,v_2)\neq (v_1,b)\ {\rm with}\ v_1=0,-1,...,1-c,&   \hbox{if $c>1$, $d=1$};\\
-c<v_1\le a,-d<v_2\le b, (v_1,v_2)\neq (0,b),&   \hbox{if $c=d=1$}.\\
  \end{array}
\right.
\end{align}

\begin{thm}\label{thm10}  Let $$\bar{B}:=\{\Lambda^{m(v)}x_1^{v_1}x_2^{v_2}: x_1^{v_1}x_2^{v_2}\in B\},$$
where $m$ denotes a function on $\mathbb{Z}^2$.
Then set $\bar{B}$ forms a basis for $H^2(\Omega^{\bullet}(\mathcal{C}_0,\nabla(D^{(\Lambda)})))$.
\end{thm}

Hence we may give a method to compute the Hodge polygon.
Using Wan's decomposition theorems and diagonal local theory we decide when the Newton polygon of
$L(\bar{F}(\bar{\lambda},x),T)^{-1}$ coincides its lower bound.
Concretely, we prove the following.
\begin{thm}\label{thm9}
If $\gcd(a,d)=1$ and $p\equiv 1 \mod ab[c,d]$, then $NP(\bar{F})$ equals to the Hodge polygon of $\bar{F}$ at $p$.
\end{thm}
It follows that the Newton polygon can be computed under the condition of Theorem \ref{thm9}.
Specially, when $c=d=1$ and $p\equiv 1 \mod ab$, the slope sequence of $L(\bar{F}(\bar{\lambda},x),T)^{-1}$ is $$\Big\{\frac{ai+bj}{ab}\Big\}_{i=0,...,b,j=0,...,a}.$$

In the point of view of Dwork's deformation theory,
both Dwork \cite{[Dwo74]} and Sperber \cite{Sp1} obtained the deformation equations
identified by Katz as the Picard-Fuchs equations \cite{NK68}.
For $g^{(1)}(\Lambda,x)$, Dwork showed that there is a strong Frobenius structure on the solution space of the deformation equation.
Inspired by this, Dwork \cite{[Dwo742]} has conjectured that Frobenius structures of differential equations exist quite widely.
Actually, Sperber's work \cite{Sp1} adds an evidence for Dwork's conjecture \cite{[Dwo742]}.

The study about the Frobenius structure of Picard-Fuchs differential equation began with Dwork \cite{Dwo69}.
Dwork explored a 4-dimensional family and identified the reciprocal zeros of its zeta function by studying the $p$-adic solutions of Picard-Fuchs equation.
Using the Griffiths-Dwork method \cite{CK99} and
the theory of GKZ system, G$\ddot{\text{a}}$hrs computed the Picard-Fuchs equations for
general one-parameter families of Calabi-Yau varieties \cite{[Gah11]}, \cite{[Gah13]}.
Combining G$\ddot{\text{a}}$hrs' result and the Frobenius structure of Picard-Fuchs equation, Doran et al. obtained the condition what two Calabi-Yau families share a common factor in their zeta functions \cite{DKS}.

Another aim of this paper is to study the Dwork's deformation equation for our family.
Let $D_{\Lambda}=\Lambda\frac{\partial}{\partial\Lambda}+\frac{\Lambda}{x_1^cx_2^d}$ be the connection.
Using the theory of GKZ system, we obtain the specific form of the deformation equation. Furthermore, we have the
following result.
\begin{thm}\label{thm11}
The differential equation satisfied by $\bar{1}:=1+\sum_{l=1}^2D_{l,\Lambda}(\mathcal{C}_0)$ is
 \begin{equation}\label{eq01}
 \Big(\frac{cD_{\Lambda}}{a}-(bc-1) \Big)\cdots \frac{cD_{\Lambda}}{a}  \Big(\frac{dD_{\Lambda}}{b}-(ad-1) \Big)\cdots \frac{dD_{\Lambda}}{b} \Big(D_{\Lambda}-(ab-1) \Big)\cdots D_{\Lambda}= \Lambda^{ab},
\end{equation}
and it is irreducible.
\end{thm}
Let $N:=ad+ab+bc=\dim H^2(\Omega^{\bullet}(\mathcal{C}_0,\nabla(D^{(\Lambda)})))$. Then $\{\bar{1},D_{\Lambda}(\bar{1}),\cdots, (D_{\Lambda})^{N-1}(\bar{1})\}$ consists a basis of
the top dimensional space
$H^2(\Omega^{\bullet}(\mathcal{C}_0,\nabla(D^{(\Lambda)})))$. The transpose of $G$, denoted by $G^t$, is the connection matrix
of $D_{\Lambda}$ on the basis $\{\bar{1},D_{\Lambda}(\bar{1}),\cdots, (D_{\Lambda})^{N-1}(\bar{1})\}$.
Clearly, the matrix $G^t$ is totally determined by (\ref{eq01}).
Let $\Omega$ be an algebraically closed field, and complete under valuation extending that of $\mathbb{Q}_p$.
We establish the strong Frobenius structures of the deformation equations using Dwork's dual theory
and deformation theory.
Our work adds some new evidences for Dwork's conjecture \cite{[Dwo742]}.
Equivalently, we prove the following result.
\begin{thm}\label{thm7}
Let $a\in \Omega$. If $y$ is a solution of
\begin{equation*}
\Lambda\frac{\partial}{\partial \Lambda}(C_0(\Lambda),C_1(\Lambda),\cdots,C_{N-1}(\Lambda))=(C_0(\Lambda),C_1(\Lambda),\cdots,C_{N-1}(\Lambda))G
\end{equation*}
near $a$, then so is $y^{\sigma\phi}U(\Lambda)$, where $U(\Lambda)$ is the matrix of Frobenius map with respect to the basis $\{\bar{1}, D_{\Lambda}(\bar{1}),\cdots, D_{\Lambda}^{N-1}(\bar{1})\}$.
\end{thm}

This paper is organized as follow. In section 2, we construct some
relative $p$-adic cohomologies
by following Haessig and Sperber.
In section 3, we study the Newton polygon of the $L$-function $L(\bar{F}(\bar{\lambda},x),T)^{-1}$. For this purpose,
we explore the specific forms of bases of top cohomology spaces and review Wan's decomposition theorems.
In section 4, we compute Dwork's deformation equation for our family and establish the strong Frobenius structure of this equation.

\section{Cohomological setup}
Recall that
$$\bar{F}(\Lambda,x):=x_1^a+x_2^b+\frac{\Lambda}{x_1^cx_2^d}\in \mathbb{F}_q[\Lambda][x_1^{\pm},x_2^{\pm }] $$ with $\bar{f}(x):=x_1^a+x_2^b$,
$\mu:=\frac{1}{x_1^cx_2^d}$ and
parameter $\Lambda$. The support of $\bar{f}$, denoted by ${\it Supp} (\bar{f})$,
is defined as
$${\rm Supp} (\bar{f}):=\{(a,0),(0,b)\}.$$
Let $\sigma:=\Delta (\bar{f})$ be the convex closure of ${\rm Supp} (\bar{f})$,
and let $\Delta_{\infty} (\bar{f})$ be the convex closure of $\Delta (\bar{f})\cup \{(0,0)\}$.
Let $Cone (\bar{f})$ be the union of rays from 0 passing through $\Delta_{\infty} (\bar{f})$.
Set $M(\bar{f}):=Cone (\bar{f})\cap \mathbb{Z}^2$. Then
$\dim \Delta_{\infty} (\bar{f})=\dim Cone (\bar{f})=2$.
The two points of ${\rm Supp} (\bar{f})$ determine the affine line
$$l_{\sigma}(v):=\Big\{(v_1,v_2) | \frac{v_1}{a}+\frac{v_2}{b}=1\Big\}.$$
Clearly, $\mu=(-c,-d)\notin M(\bar{f})$ and $l_{\sigma}(\mu)<1$.
Note that $p\nmid cd$ and $\bar{f}$ is nondegenerate with respect to $\Delta_{\infty}(\bar{f})$. It follows from Theorem 2.1 of \cite{HS17}
that for each $\bar{\lambda}\in \bar{\mathbb{F}}_q^{*}$,
$\bar{F}(\bar{\lambda},x)$ is nondegenerate with respect to $\Delta_{\infty}(\bar{f},\mu)$,
where $\Delta_{\infty}(\bar{f},\mu)$ is the convex closure of $\Delta_{\infty}(\bar{f})\cup \{\mu\}$.

Let $Cone(\bar{f},\mu)$ be the cone in $\mathbb{R}^2$ over $\Delta_{\infty}(\bar{f},\mu)$ and let
$M(\bar{f},\mu):=Cone(\bar{f},\mu)\cap \mathbb{Z}^2$. Notably, $\Delta_{\infty}(\bar{f},\mu)$ is the triangle with vertices
$(a,0),(0,b)$ and $\mu$, and $M(\bar{f},\mu)=\mathbb{Z}^2$. If $\tau$ is a closed face of $\Delta_{\infty}(\bar{f},\mu)$ not containing 0, then we say $\tau$ is {\it a face at $\infty$}. The three sides of the triangle are three faces of codimension one at $\infty$ of $\Delta_{\infty}(\bar{f},\mu)$.

For $\bar{\lambda}\in \bar{\mathbb{F}}_q^{*}$, we let $\mathbb{F}_q^{(\bar{\lambda})}:=\mathbb{F}_q(\bar{\lambda})$
denote the field generated over $\mathbb{F}_q$ by $\bar{\lambda}$. For $v\in M(\bar{f},\mu)$, let $w(v)$
denote the smallest nonnegative rational number $\bar{e}$ such that $v\in \bar{e}\cdot \Delta_{\infty}(\bar{f},\mu)$.
There is a positive integer $e$ such that $w(M(\bar{f},\mu))\subset (1/e)\mathbb{Z}_{\ge 0}$.
Let $R^{(\bar{\lambda})}:=\mathbb{F}_q^{(\bar{\lambda})}[M(\bar{f},\mu)]$.
Then $R^{(\bar{\lambda})}$ has an increasing filtration indexed by $(1/e)\mathbb{Z}_{\ge 0}$.
Let $\bar{R}^{(\bar{\lambda})}$ denote the associated graded ring.
Thus we construct two complexes as follows.
The spaces in both complexes are the same
$$\Omega^{i}(\bar{R}^{(\bar{\lambda})},\nabla(\bar{F}(\bar{\lambda},x))):=\Omega^{i}(\bar{R}^{(\bar{\lambda})},\nabla(D^{(\bar{\lambda})})):=\bigoplus_{1\le j_1<j_i\le 2} \bar{R}^{(\bar{\lambda})} \frac{dx_{j_1}}{x_{j_1}}\wedge \frac{dx_{j_i}}{x_{j_i}}$$
with respective boundary operators given by
$$\nabla(\bar{F}(\bar{\lambda},x)\Big( \zeta \frac{dx_{j_1}}{x_{j_1}}\wedge \frac{dx_{j_i}}{x_{j_i}}\Big):=\Big(\sum_{l=1}^nx_l\frac{ \partial \bar{F}(\bar{\lambda},x)}{\partial x_l}\zeta \frac{dx_{l}}{x_{l}}\Big)\wedge\frac{dx_{j_1}}{x_{j_1}}\wedge \frac{dx_{j_i}}{x_{j_i}}$$
and
$$\nabla(D^{(\bar{\lambda})})\Big( \zeta \frac{dx_{j_1}}{x_{j_1}}\wedge \frac{dx_{j_i}}{x_{j_i}}\Big):=\Big(\sum_{l=1}^n(D_{l}^{(\bar{\lambda})}\zeta ) \frac{dx_{l}}{x_{l}}\Big)\wedge\frac{dx_{j_1}}{x_{j_1}}\wedge \frac{dx_{j_i}}{x_{j_i}},$$
where $$D_l^{(\bar{\lambda})}:=x_l\frac{\partial}{\partial x_l}+x_l\frac{ \partial \bar{F}(\bar{\lambda},x)}{\partial x_l},\ l=1,2. $$

Since $\bar{F}(\bar{\lambda},x)$ is nondegenerate with respect to $\Delta_{\infty}(\bar{f},\mu)$, we have
\begin{thm}[Theorem 2.2, \cite{HS17}]\label{thm0}
For every choice $\bar{\lambda}\in \bar{\mathbb{F}}_q^{*}$, complexes $\Omega^{\bullet}(\bar{R}^{(\bar{\lambda})},\nabla(\bar{F}(\bar{\lambda},x)))$
and $\Omega^{\bullet}(\bar{R}^{(\bar{\lambda})}, \nabla(D^{(\bar{\lambda})}))$
are acyclic except in the top dimension 2. The top dimensional cohomology $H^2$ is a finite free $\mathbb{F}_q^{(\bar{\lambda})}$-algebra
of rank $2!Vol(\Delta_{\infty}(\bar{f},\mu))$.
For each $i\in (1/e)\mathbb{Z}_{\ge 0}$, we may choose a monomial basis $B_i^{(\bar{\lambda})}$
consisting of monomials of weight $i$ for an $\mathbb{F}_q^{(\bar{\lambda})}$-vector space $V_i^{(\bar{\lambda})}$
such that the $i$-th graded piece $\bar{R}_i^{(\bar{\lambda})}$ of $\bar{R}^{(\bar{\lambda})}$ may be written as
$$\bar{R}_i^{(\bar{\lambda})}=V_i^{(\bar{\lambda})}\oplus \sum_{l=1}^2x_l\frac{\partial \bar{F}(\bar{\lambda},x)}{\partial x_l}\bar{R}_{i-1}^{(\bar{\lambda})}$$
so that if $B^{(\bar{\lambda})}=\cup_{i\in (1/e)\mathbb{Z}_{\ge 0}}B_i^{(\bar{\lambda})}$
and $V^{(\bar{\lambda})}=\sum_{i\in (1/e)\mathbb{Z}_{\ge 0}}V_i^{(\bar{\lambda})}$ is the $\mathbb{F}_q^{(\bar{\lambda})}$-vector space with
basis $B^{(\bar{\lambda})}$,
then
$$\bar{R}^{(\bar{\lambda})}=V^{(\bar{\lambda})}\oplus \sum_{l=1}^2x_l \frac{\partial \bar{F}(\bar{\lambda},x)}{\partial x_l}\bar{R}^{(\bar{\lambda})} $$
and
 $$\bar{R}^{(\bar{\lambda})}=V^{(\bar{\lambda})}\oplus \sum_{l=1}^2D_l^{(\bar{\lambda})}\bar{R}^{(\bar{\lambda})}.  $$
\end{thm}

{\bf Remark.} As showed in \cite{HS17}, the basis $B^{(\bar{\lambda})}$ is actually independent of the choice of $\bar{\lambda}\in\bar{\mathbb{F}}_q^{*}$.



Now we express the specific form of the weight function $w$ on the closed subcones of $Cone(\bar{f},\mu)$ corresponding to the codimension one faces
$\omega$ of $\Delta_{\infty}(\bar{f},\mu)$ at $\infty$.
If $\omega=\sigma=\Delta(\bar{f})$ and $v\in 
M(\bar{f})$,
then let
\begin{equation}\label{eq1}
w(v):=l_{\sigma}(v).
\end{equation}
For $\tau\in {\rm Supp}(\bar{f})$, let $C(\tau,\mu)$ denote the segment connecting $\tau$ and $\mu$, and $Cone(\tau,\mu)$ denote the set of rays from origin passing through $C(\tau,\mu)$.
If $\tau=(a,0)$, $\omega=C(\tau,\mu)$ and $v\in M(\tau,\mu):=Cone(\tau,\mu)\cap \mathbb{Z}^2 $,
then define
\begin{equation}\label{eq11}
w(v):= -\frac{(a+c)v_2}{ad}+\frac{v_1}{a}.
\end{equation}
If $\tau=(0,b)$, $\omega=C(\tau,\mu)$ and $v\in M(\tau,\mu)$,
then define
\begin{equation}\label{eq111}
w(v):= -\frac{(b+d)v_1}{bc}+\frac{v_2}{b}.
\end{equation}

\subsection{Total space}
To construct relative cohomology, we view $\Lambda$ as a variable. Hence $\bar{F}(\Lambda,x)$
is viewed as a Laurent polynomial in $3$ variables.
For convenience, we denote the ordering of coordinates in $\mathbb{R}^{3}$ associated to the monomial
$\Lambda^rx^v$ to be the point $(r;v_1,v_2)\in \mathbb{R}^{3}$. Then
$${\rm Supp}(\bar{F})=\{ (1;\mu)\}\cup \{ (0;v)|v\in {\rm Supp}(\bar{f})\}.$$
Since $\bar{F}$ is quasi-homogeneous, all elements of ${\rm Supp}(\bar{F})$
lies on the affine hyperplane $W(r;v)=1$,
where
\begin{equation}\label{eq25}
W(r;v):=l_{\sigma}(v)+r(1-l_{\sigma}(\mu))=\frac{v_1}{a}+\frac{v_2}{b}+r(1+\frac{c}{a}+\frac{d}{b}).
\end{equation}
Define the weight of $\Lambda^rx^v$ by $W(r;v)$.

We now give the explicit form of the $m$-function in Theorem \ref{thm10}, let
\begin{align}\label{eq12}
&m(v):=\left\{
  \begin{array}{ll}
  -\frac{v_1}{c},&  \hbox{for $v\in Cone(\tau,\mu)$ and $\tau=(0,b)$};\\
    -\frac{v_2}{d},& \hbox{for $v\in Cone(\tau,\mu)$ and $\tau=(a,0)$};\\
  0,&   \hbox{for $v\in Cone(\bar{f})$}.
  \end{array}
\right.
\end{align}
The $m$-function has the following property.
\begin{lem}\label{lem3}\cite{HS17} For any $u,v\in M(\bar{f},\mu)$, one has that
$$m(u+v)\le m(u)+m(v).$$
\end{lem}

Define the set $\tilde{M}(\bar{F})$ to be
\begin{equation}\label{eq4}
\tilde{M}(\bar{F}):=\{ (r;v)\in (1/D)\mathbb Z_{\ge 0} \times \mathbb{Z}^2| v\in Cone (\bar{f},\mu),r\ge m(v)\},
\end{equation}
where $D=cd$. Let $W$ be the total weight function on $\tilde{M}(\bar{F})$.
Let $\tilde{e}$ be the smallest positive integer, divisible by $D$, such that $W(\tilde{M}(\bar{F}))\subset (1/\tilde{e})\mathbb{Z}_{\ge 0}$.
 Let $T:=\mathbb{F}_q[\tilde{M}(\bar{F})]$ be the graded $\mathbb{F}_q$-algebra with its grading given by the total
 weight function $W$
and indexed by $(1/\tilde{e})\mathbb{Z}_{\ge 0}$. We conclude from (\ref{eq4}) that $T:=\mathbb{F}_q[\tilde{M}(\bar{F})]$ is a free $S:=\mathbb{F}_q[\Lambda^{1/D}]$-algebra
with basis $\{\Lambda^{m(v)}x^v\}_{v\in M(\bar{f},\mu)}$.
It follows from (\ref{eq1}), (\ref{eq11}), (\ref{eq111}), (\ref{eq25}) and (\ref{eq12}) that
\begin{align*}
W(m(v);v)=w(v).
\end{align*}
Let $W_{\Lambda}(\Lambda^r):=r(1-l_{\sigma}(\mu))=r(1+\frac{c}{a}+\frac{d}{b})$.
Then $W(\Lambda^rx^v)=w(v)+W_{\Lambda}(\Lambda^{r-m(v)})$.

Now we construct a complex of $S$-algebras $\Omega^{\bullet}(T,\nabla(D))$ as \cite{HS17}.
Let
$$\Omega^i:= \bigoplus_{1\le j_1<j_i\le 2}T\frac{dx_{j_1}}{x_{j_1}}\wedge\frac{dx_{j_i}}{x_{j_i}},$$
with boundary operator defined by
$$ \nabla (D)\Big(\zeta \frac{dx_{j_1}}{x_{j_1}}\wedge \frac{dx_{j_i}}{x_{j_i}}\Big):= \Big(\sum_{l=1}^2 D_l(\Lambda,x)\zeta \frac{dx_l}{x_l}\Big)\wedge\frac{dx_{j_1}}{x_{j_1}}\wedge\frac{dx_{j_i}}{x_{j_i}},$$
where
$$D_l(\Lambda,x):=x_l\frac{\partial}{\partial x_l}+x_l\frac{\partial \bar{F}(\Lambda,x)}{\partial x_l},\ l=1,2.$$
For our family, we compute
$$2!Vol \Delta_{\infty}(\bar{f},\mu)=ad+ab+bc=N.$$
The following result is a special case of Haessig and Sperber's theorem.
\begin{thm}\label{thm1}(Theorem 2.5,\cite{HS17})
The complex $\Omega^{\bullet}(T,\nabla(D))$ of $S$-algebras is acyclic except in top dimension 2. $H^2(\Omega^{\bullet}(T,\nabla(D)))$
is a free filtered $S$-algebra of finite rank $N$.
If we fix $\bar{\lambda}\in \mathbb{F}_q^{*}$ (for simplicity)
and $\tilde{V}$ the free $S$-module with basis $\bar{\bar{B}}:=\{\Lambda^{m(v)}x^v: v\in B^{(\bar{\lambda})}\}$, then
$$T=\tilde{V} \oplus \sum _{l=1}^2D_l T.$$
\end{thm}

\subsection{$p$-adic Theory}
Let $p>2$.
Let $\mathbb{Q}_p$ be the field of $p$-adic number, and $\mathbb{Z}_p$ be the ring of $p$-adic integers.
Let $\mathbb{Q}_q$ be the unramified extension of $\mathbb{Q}_p$ of degree $\tilde{a}$,
and $\mathbb{Z}_q$ be its ring of integers.
Let $\zeta_p$ be a primitive $p$-th root of unity. Then $\mathbb{Z}_q[\zeta_p]$ and $\mathbb{Z}_p[\zeta_p]$
are rings of integers of $\mathbb{Q}_q[\zeta_p]$ and $\mathbb{Q}_p[\zeta_p]$, respectively.
Recall $\pi$ is an element in an algebraic closure of $\mathbb{Q}_p$ such that $\pi^{p-1}=-p$.
By Krasner's lemma, we have $ \mathbb{Q}_p(\pi)=\mathbb{Q}_p(\zeta_p)$.
Adjoining an appropriate root of $\pi$, say $\tilde{\pi}$,
we obtain totally ramified extensions of $\mathbb{Q}_q(\zeta_p)$ and
$\mathbb{Q}_p(\zeta_p)$, which are denoted by $K$ and $K_0$, respectively. Let $\mathbb{Z}_q[\tilde{\pi}]$ and
$\mathbb{Z}_p[\tilde{\pi}]$ denote the respective rings of integers of $K$ and $K_0$.

Define set $\mathcal{O}_0$ by
$$\mathcal{O}_0:=\Big\{\sum_{r=0}^{\infty}C(r)\Lambda^{r/D}\pi^{W_{\Lambda}(r/D)}:
C(r)\in \mathbb{Z}_q[\tilde{\pi}],C(r)\rightarrow 0\ {\rm as}\ r \rightarrow \infty  \Big\}$$
with a valuation via
$$\Big|\sum_{r=0}^{\infty}C(r)\Lambda^{r/D}\pi^{W_{\Lambda}(r/D)}\Big|:=\sup_{r\ge 0}\{|C(r)|\}.$$
The reduction map $\mod \tilde{\pi}$ maps $\mathcal{O}_0$ onto $S=\mathbb{F}_q[\Lambda^{1/D}]$ by sending
$$ \sum_{r=0}^{\infty}C(r)\Lambda^{r/D}\pi^{W_{\Lambda}(r/D)}\mapsto\sum_{r=0}^{\infty}\bar{C}(r)\Lambda^{r/D}.$$
Then the reduction map identifies the $\mathbb{F}_q$-algebras
$$\mathcal{O}_0/ \tilde{\pi} \mathcal{O}_0\cong S.$$
We now express the $p$-adic Banach space $\mathcal{C}_0$ concretely.
Let $\gamma$ be a positive real number. Then we define
$$\mathcal{C}_0(\gamma):=\Big\{\sum_{v\in M(\bar{f},\mu)}\zeta(v)\pi^{\gamma w(v)}\Lambda^{m(v)}x^v : \zeta(v)\in \mathcal{O}_0,\zeta(v)
\rightarrow 0\ {\rm as}\ w(v) \rightarrow \infty \Big\}$$
to be a $p$-adic Banach $\mathcal{O}_0$-algebra. Especially, we write $\mathcal{C}_0$ for $\mathcal{C}_0(1)$.
Then the reduction map $\mod \tilde{\pi}$ taking
$$\sum_{v\in M(\bar{f},\mu)}\zeta(v)\pi^{w(v)}\Lambda^{m(v)}x^v \mapsto \sum_{v\in M(\bar{f},\mu)}\bar{\zeta}(v)\Lambda^{m(v)}x^v,$$
induces an isomorphism of $S$-algebras
$$\mathcal{C}_0/\tilde{\pi}\mathcal{C}_0\cong T.$$

We are now in a position to construct a complex of $p$-adic spaces.
Let
$$\theta(t):=\exp \pi(t-t^p).$$
If we write
$\theta(t)=\sum_{i=0}^{\infty}\lambda_it^i$, it then follows from \cite{[Dwo62]} that
\begin{equation}\label{eq00}
{\rm ord}_p \lambda_i\ge \frac{p-1}{p^2}\cdot i
\end{equation}
for every $i\ge 0$.


Note that $F(\Lambda,x)$ has the total weight $W\le 1$. Then multiplication by
$x_l \frac{\partial F^{(0)}}{\partial x_l}(l=1,2)$ defines an endomorphism of $\mathcal{C}_0$.
Hence we may define a complex of $\mathcal{O}_0$-modules $\Omega^{\bullet}(\mathcal{C}_0,\nabla(D^{(\Lambda)}))$ by letting
$$\Omega^{i}(\mathcal{C}_0,\nabla(D^{(\Lambda)})):=\bigoplus_{1\le j_1<j_i\le 2} \mathcal{C}_0 \frac{d x_{j_1}}{x_{j_1}}\wedge \frac{d x_{j_i}}{x_{j_i}}$$
with boundary map
$$\nabla(D^{(\Lambda)})(\zeta \frac{d x_{j_1}}{x_{j_1}}\wedge \frac{d x_{j_i}}{x_{j_i}})=\Big(\sum_{l=1}^2D_{l,\Lambda}(\zeta)\frac{d x_{l}}{x_{l}}\Big)\wedge \frac{d x_{j_1}}{x_{j_1}}\wedge\frac{d x_{j_i}}{x_{j_i}},$$
where $$D_{l,\Lambda}=x_l\frac{\partial}{\partial x_l}+x_l\frac{\partial F^{(0)}}{\partial x_l}.$$
One also has that
\begin{equation}\label{eq13}
D_{l,\Lambda}=\frac{1}{\exp F^{(0)}}\circ x_l\frac{\partial}{\partial x_l}\circ \exp F^{(0)}.
\end{equation}
Using the same argument to Theorem \ref{thm1} as \cite{HS17}, \cite{HS14}, or going back to \cite{AS89} or \cite{[PM70]}
we have the following result.
\begin{thm}(Theorem 3.1, \cite{HS17})\label{thm5}
The complex $\Omega^{\bullet}(\mathcal{C}_0,\nabla(D^{(\Lambda)}))$ is acyclic except in top dimension 2 and $H^2(\Omega^{\bullet}(\mathcal{C}_0,\nabla(D^{(\Lambda)})))$
is a free $\mathcal{O}_0$-module of rank equal to $N$.
Furthermore,
$$\mathcal{C}_0=\sum_{(m(v);v)\in \bar{\bar{B}}}\mathcal{O}_0\pi^{w(v)}\Lambda^{m(v)}x^v \oplus \sum _{l=1}^2D_{l,\Lambda}(\mathcal{C}_0).$$
\end{thm}

\subsection{Frobenius map}

Set $$\alpha_0:=\sigma^{-1}\circ \frac{1}{\exp F^{(0)}(\Lambda^p,x)}\circ \psi_p\circ \exp F^{(0)}(\Lambda,x)$$
 and
$$\alpha:=\frac{1}{\exp F^{(0)}(\Lambda^q,x)}\circ \psi_q\circ \exp F^{(0)}(\Lambda,x),$$
where $\psi_p$ and $\psi_q$ are defined as
$$\psi_p(\sum A(v)x^v)=\sum A(pv)x^v$$
$$\psi_q(\sum A(v)x^v)=\sum A(qv)x^v,$$
and $\sigma\in Gal(\mathbb{Q}_q(\zeta_p)/\mathbb{Q}_p(\zeta_p))$ is the Frobenius automorphism of $Gal(\mathbb{Q}_q/\mathbb{Q}_p)$
extended to $K$ by requiring $\sigma(\tilde{\pi})=\tilde{\pi}$ and $\sigma(\zeta_p)=\zeta_p$.

By (\ref{eq13})
the following communication laws hold
\begin{equation}\label{eq21}
q D_{l,\Lambda^q}\circ \alpha=\alpha\circ D_{l,\Lambda}\
{\rm and}\ p D_{l,\Lambda^p}\circ \alpha_0=\alpha_0\circ D_{l,\Lambda}
\end{equation}
for $l=1,2$.
Since the communication laws hold up to the change from $\Lambda$ to $\Lambda^q$, this motivates us to introduce some new spaces as Haessig and Sperber \cite{HS17}. To construct new spaces, we view $\Lambda^q$ as one variable instead of $\Lambda$.
We view $\Lambda^r=(\Lambda^q)^{r/q}$ for $r\in (1/D)\mathbb{Z}_{\ge 0}$. Then define
$$W_{\Lambda^q}(\Lambda^r):=(r/q)(1-l_{\sigma}(\mu))=W_{\Lambda}(\Lambda^{r/q}).$$
Hence we define the monoid $\tilde{M}_q(\bar{F})$ analogously to (\ref{eq4}) by
$$\tilde{M}_q(\bar{F}):=\{ (r;v)\in (1/D)\mathbb{Z}_{\ge 0}\times \mathbb{Z}^2: v\in Cone(\bar{f},\mu), r\ge qm(v)\}. $$
Define for $q=p^{\tilde{a}}$ with $\tilde{a}\in \mathbb{Z}_{\ge 0}$
$$\mathcal{O}_{0,q}:=\Big\{\sum_{r=0}^{\infty}A(r)\Lambda^{r/D}\pi^{W_{\Lambda^q}(r/D)}:
A(r)\in \mathbb{Z}_q[\tilde{\pi}], A(r)\rightarrow 0\ {\rm as}\ r \rightarrow \infty  \Big\}$$
with valuation
$$\Big|\sum_{r=0}^{\infty}A(r)\Lambda^{r/D}\pi^{W_{\Lambda^q}(r/D)}\Big|:=\sup_{r\ge 0}\{|A(r)|\}.$$
For positive real number $\gamma$, let
$$\mathcal{C}_{0,q}(\gamma):=\Big\{\sum_{v\in M(\bar{f},\mu)}\zeta(v)\pi^{\gamma w(v)}\Lambda^{qm(v)}x^v : \zeta(v)\in \mathcal{O}_{0,q},\zeta(v)
\rightarrow 0\ {\rm as}\ w(v) \rightarrow \infty \Big\}$$
be a $p$-adic Banach space with valuation
 $$\Big|\sum_{v\in M(\bar{f},\mu)}\zeta(v)\pi^{\gamma w(v)}\Lambda^{qm(v)}x^v\Big|:=\sup_{v\in M(\bar{f},\mu)}\{ |\zeta(v)|\}.$$
We write $\mathcal{C}_{0,q}$ for $\mathcal{C}_{0,q}(1)$.
For $(r;v)\in \tilde{M}_q(\bar{F})$, the total weight $W_q$ is defined as
$$W_q(r;v):=W_{\Lambda^q}(r-qm(v))+w(v).$$
Then the reduction map $\mod \tilde{\pi}$ acts on $\mathcal{C}_{0,q}$ by taking
$$\sum_{v\in M(\bar{f},\mu)}\sum_{r=0}^{\infty}A(r)\pi^{W_q(r,v)}\Lambda^{r}x^v \mapsto \sum_{v\in M(\bar{f},\mu)}\sum_{r=0}^{\infty}\bar{A}(r)\Lambda^{r}x^v.$$
Then Haessig and Sperber proved the following result.
\begin{thm}\label{thm2}(Theorem 3.2, \cite{HS17})
Let $D_{l,\Lambda^q}=x_l\frac{\partial}{\partial x_l}+x_l\frac{\partial F^{(0)}(\Lambda^q,x)}{\partial x_l}$.
Let $\Omega^{\bullet}(\mathcal{C}_{0,q}, \nabla(D^{(\Lambda^q)}))$ be the complex
$$\Omega^i:=\bigoplus _{1\le j_1<j_i\le 2}\mathcal{C}_{0,q} \frac{d x_{j_1}}{x_{j_1}}\wedge\frac{d x_{j_i}}{x_{j_i}}$$
with boundary map
$$ \nabla (D^{(\Lambda^q)})(\zeta \frac{d x_{j_1}}{x_{j_1}}\wedge \frac{d x_{j_i}}{x_{j_i}})=\Big(\sum_{l=1}^2 D_{l,\Lambda^q}(\zeta)\frac{d x_{l}}{x_{l}}\Big)\wedge
\frac{d x_{j_1}}{x_{j_1}}\wedge\frac{d x_{j_i}}{x_{j_i}}.$$
This complex is acyclic except in top dimension 2 and $H^2(\Omega^{\bullet}(\mathcal{C}_{0,q}, \nabla(D^{(\Lambda^q)}) ))$
is a free $\mathcal{O}_{0,q}$-module of rank equals to $N$.
Furthermore,
$$\mathcal{C}_{0,q}=\sum_{v\in \tilde{B}} \mathcal{O}_{0,q} \pi^{w(v)}\Lambda^{qm(v)}x^v \oplus
\sum_{l=1}^2 D_{l,\Lambda^q}\mathcal{C}_{0,q}$$
where $\tilde{B}:=\{v:(m(v);v)\in \bar{\bar{B}}\}$.
\end{thm}

For $0<\gamma< \frac{(p-1)^2}{p}$,
it follows from (\ref{eq00}) that
$$\mathfrak{F}(\Lambda,x):= \theta(\Lambda x^{\mu})\theta (x_1^a)\theta(x_2^b)\in \mathcal{C}_0\Big(\frac{(p-1)^2}{p^2}\Big)\subset
\mathcal{C}_0\Big(\frac{\gamma}{p}\Big).$$
Let $$\mathfrak{F}_{\tilde{a}}(\Lambda,x):=\prod_{i=0}^{\tilde{a}-1}\mathfrak{F}^{\sigma^i}(\Lambda^{p^i},x^{p^i}).$$
Then
$$\mathfrak{F}_{\tilde{a}}(\Lambda,x)\in \mathcal{C}_0\Big(\frac{\gamma}{q}\Big).$$
Hence for $0<\gamma< \frac{(p-1)^2}{p}$, we have
$$\mathcal{C}_0(\gamma)\subseteq \mathcal{C}_0\Big(\frac{\gamma}{p}\Big)\stackrel{ \mathfrak{F}(\Lambda,x)}\longrightarrow \mathcal{C}_0\Big(\frac{\gamma}{p}\Big)\stackrel{ \psi_p}\longrightarrow \mathcal{C}_{0,p}(\gamma) .$$
We can check that
$$ \alpha_0=\sigma^{-1}\circ \psi_p\circ \mathfrak{F}(\Lambda,x) \ {\rm and}\ \alpha=\psi_q\circ \mathfrak{F}_{\tilde{a}}(\Lambda,x).$$
Note that $p>2$. Then $\frac{(p-1)^2}{p}>1$.
 Then we see that $\alpha_0$ maps $\sigma^{-1}$-semilinearly $\mathcal{C}_0$
into $\mathcal{C}_{0,p}.$ Similarly, $\alpha$ maps $\mathcal{C}_0$
into $\mathcal{C}_{0,q}$ linearly over $\mathbb{Z}_q[\tilde{\pi}]$.
Hence we may use $\alpha_0$ and $\alpha$ to define chain maps as follows.

Let
\begin{equation}\label{eq14}
Frob_{\Lambda}^i:= \bigoplus _{1\le j_1<j_i\le 2} q^{2-i}\alpha \frac{d x_{j_1}}{x_{j_1}}\wedge\frac{d x_{j_i}}{x_{j_i}},
Frob_{0,\Lambda}^i:= \bigoplus _{1\le j_1<j_i\le 2} p^{2-i}\alpha_0 \frac{d x_{j_1}}{x_{j_1}}\wedge \frac{d x_{j_i}}{x_{j_i}}.
\end{equation}
The commutation rules (\ref{eq21}) ensure that (\ref{eq14}) defines chain maps
$$\Omega^{\bullet}(\mathcal{C}_{0},\nabla (D^{(\Lambda)}))\stackrel{Frob_{0,\Lambda}}{\longrightarrow} \Omega^{\bullet}(\mathcal{C}_{0,p},\nabla (D^{(\Lambda^p)}))$$
and
$$\Omega^{\bullet}(\mathcal{C}_{0},\nabla (D^{(\Lambda)}))\stackrel{Frob_{\Lambda}}{\longrightarrow} \Omega^{\bullet}(\mathcal{C}_{0,q},\nabla (D^{(\Lambda^q)})).$$

Let $\bar{\lambda}\in \bar{\mathbb{F}}_q^{*}$ with $\deg (\bar{\lambda})=[\mathbb{F}_q(\bar{\lambda}):\mathbb{F}_q]$.
We define an additive character $\Theta:\mathbb{F}_q\rightarrow \bar{\mathbb{Q}}_p$ by $\Theta:=\theta(1)^{Tr_{\mathbb{F}_q/\mathbb{F}_p}(\cdot)}$
and $\Theta_{\bar{\lambda}}:=\Theta \circ Tr_{\mathbb{F}_q(\bar{\lambda})/\mathbb{F}_p}(\cdot) $.
We consider the toric exponential sum
$$S_k(\bar{F},\bar{\lambda}):=\sum_{x\in (\mathbb{F}_{q^{k\deg (\bar{\lambda})}}^{*})^2} \Theta_{\bar{\lambda}} \circ Tr_{\mathbb{F}_{q^{k\deg(\bar{\lambda})}}/\mathbb{F}_{q}(\bar{\lambda}) }(\bar{F}(\bar{\lambda},x))$$
and the associated $L$-function
$$L(\bar{F}(\bar{\lambda},x),T):=\exp\Big(\sum_{k=1}^{\infty} S_k(\bar{F},\bar{\lambda})\frac{T^k}{k}\Big).$$

Let $\lambda$ be the Teichum${\rm \ddot{u}}$ller representative of $\bar{\lambda}$. Let $\mathcal{O}_{0,\lambda}:=\mathbb{Z}_q[\tilde{\pi},\lambda]$.
Let $sp_{\lambda}$ be the specialization map at $\lambda$, from $\mathcal{O}_0$ to $\mathcal{O}_{0,\lambda}$ induced by the map
sending $\Lambda\mapsto \lambda$. Let $\mathcal{C}_{0,\lambda}$ be the $\mathcal{O}_{0,\lambda}$-module obtained by specializing the space
$\mathcal{C}_{0}$ at $\Lambda=\lambda$. The complex $\Omega^{\bullet}(\mathcal{C}_{0,\lambda},\nabla(D^{(\lambda)}))$ is defined
as the complex $\Omega^{\bullet}(\mathcal{C}_{0},\nabla(D^{(\Lambda)}))$ but with $\mathcal{C}_0$ replaced by $\mathcal{C}_{0,\lambda}$
and $D_{l,\Lambda}=x_l\frac{\partial}{\partial x_l}+\pi x_l \frac{\partial F(\Lambda,x)}{\partial x_l}$
replaced by $D_{l,\lambda}=x_l\frac{\partial}{\partial x_l}+\pi x_l \frac{\partial F(\lambda,x)}{\partial x_l}$.
Furthermore, we define $\mathfrak{F}(\lambda,x):=sp_{\lambda}\mathfrak{F}(\Lambda,x)$,
$\mathfrak{F}_{\tilde{a}}(\lambda,x):=sp_{\lambda}\mathfrak{F}_{\tilde{a}}(\Lambda,x)$ and set
$\alpha_{0,\lambda}:=\sigma^{-1}\circ \psi_p\circ \mathfrak{F}(\lambda,x)$
and $\alpha_{\lambda}:=\alpha_{0,\lambda}^{\tilde{a}\deg(\bar{\lambda})}$.
We define $Frob_{\lambda}^{\bullet}$ acting as a chain map on $\Omega^{\bullet}(C_{0,\lambda}, \nabla(D^{(\lambda)}))$
as in (\ref{eq14}) but with $\alpha$ replaced by $\alpha_{\lambda}$
and $q$ replaced by $q^{\deg{(\bar{\lambda})}}$. Then by the Dwork's trace formula, we have
\begin{align*}
S_k(\bar{F},\bar{\lambda})&=(q^{k\deg{(\bar{\lambda})}}-1)^2Tr(\alpha_{\lambda}\ |\ \mathcal{C}_{0,\lambda})\\
&=\sum_{i=0}^2(-1)^iTr(H^i(Frob_{\lambda})^k\ |\ H^i(\mathcal{C}_{0,\lambda}, \nabla (D^{(\lambda)}))).
\end{align*}

It has been proved in \cite{AS89} that the cohomology of $\Omega^{\bullet}(C_{0,\lambda}, \nabla(D^{(\lambda)}))$ is acyclic except in top dimension $2$, then we have
$$S_k(\bar{F},\bar{\lambda})= Tr\big(H^2 (Frob_{\lambda})^k| H^2(\mathcal{C}_{0,\lambda}, \nabla (D^{(\lambda)}))\big).$$
Hence
$$L(\bar{F}(\bar{\lambda},x), T)^{-1}=\det (1-Frob_{\lambda}T|H^2(\mathcal{C}_{0,\lambda}, \nabla (D^{(\lambda)}))).$$

It follows from \cite{AS89} that for each $\bar{\lambda}\in \bar{\mathbb{F}}_q^{*}$, the Newton polygon of $L(\bar{F}(\bar{\lambda},x), T)^{-1}$
lies over the Newton polygon using $ord_{q^{\deg(\bar{\lambda})}}$ of
\begin{equation}\label{eq27}
\prod_{v\in \tilde{B}}(1-(q^{\deg(\bar{\lambda})})^{w(v)}T),
\end{equation}
where $\tilde{B}$ is defined as Theorem \ref{thm2}.

\section{Lower bound for Newton polygon}

\subsection{Basis}
Recall that $B$ is the set of vectors whose coordinates are integers satisfying condition (\ref{eq02}).
We have that
\begin{lem}\label{lem1}
Fix $\bar{\lambda}\in \mathbb{F}_q^{*}$. Let monomial $x_1^{v_1}x_2^{v_2}$ such that $v_1,v_2\in \mathbb{Z}$ and $-c<v_1\le a,-d<v_2\le b$. Then
$x_1^{v_1}x_2^{v_2}$ is a linear combination of elements in $B$ over $\mathbb{F}_q$ modulo $\sum_{l=1}^2D_{l}^{(\bar{\lambda})}\bar{R}^{(\bar{\lambda})}$.
\end{lem}
\begin{proof}
We only prove the case that $c>1$ and $d>1$ since other cases can be done in a similar and simpler way.
In what follows, we let $c>1$ and $d>1$. If $(v_1,v_2)$ satisfies
$$\frac{d-1}{c-1}(v_1-a)\le v_2< \frac{d-1}{c-1}v_1+b,$$
then $x_1^{v_1}x_2^{v_2}\in B$. Hence it remains to show that Lemma \ref{lem1} holds for $(v_1,v_2)$ satisfies 
$$v_2<\frac{d-1}{c-1}(v_1-a)\ {\rm or}\ v_2\ge \frac{d-1}{c-1}v_1+b.$$
Here we only prove the case $v_2<\frac{d-1}{c-1}(v_1-a)$ since the other case can be handled in a similar way.
We have the following claim.

{\it Claim. Suppose that $(u_1,u_2)\in \mathbb{Z}^2$ satisfies
\begin{equation}\label{eq000}
-c<u_1\le a, -d<u_2\le b,u_2<\frac{d-1}{c-1}(u_1-a).
\end{equation}
Then one has:\\
(1). Monomial $x_1^{u_1}x_2^{u_2}$ is a linear combination of $x_1^{u_1-a}x_2^{u_2}$
and $x_1^{u_1-a}x_2^{u_2+b}$ modulo $\sum_{l=1}^2D_{l}^{(\bar{\lambda})}\bar{R}^{(\bar{\lambda})}$.\\
(2). If $x_1^{u_1-a}x_2^{u_2}\in B$, then
$x_1^{u_1-a}x_2^{u_2+b}\in B $.\\
(3). If $x_1^{u_1-a}x_2^{u_2}\notin B$, then $(u_1-a,u_2)$ satisfies
(\ref{eq000}). If $x_1^{u_1-a}x_2^{u_2+b}\notin B$, then $(u_1-a,u_2+b)$ satisfies
(\ref{eq000}).}

Using operators $D_{1}^{(\bar{\lambda})}$ and $D_{2}^{(\bar{\lambda})}$ to act on
$x_1^{u_1-a}x_2^{u_2}$, we obtain that
$$D_{1}^{(\bar{\lambda})}(x_1^{u_1-a}x_2^{u_2})=(u_1-a)x_1^{u_1-a}x_2^{u_2}+ax_1^{u_1}x_2^{u_2}-c\bar{\lambda} x_1^{u_1-a-c}x_2^{u_2-d}$$
and
$$D_{2}^{(\bar{\lambda})}(x_1^{u_1-a}x_2^{u_2})=u_2x_1^{u_1-a}x_2^{u_2}+bx_1^{u_1-a}x_2^{u_2+b}-d\bar{\lambda} x_1^{u_1-a-c}x_2^{u_2-d}.$$
Thus
\begin{equation*}\label{eq0}
x_1^{u_1}x_2^{u_2}\equiv \frac{1}{a}(c\bar{\lambda} x_1^{u_1-a-c}x_2^{u_2-d}-(u_1-a)x_1^{u_1-a}x_2^{u_2}) \mod D_{1}^{(\bar{\lambda})}(x_1^{u_1-a}x_2^{u_2})
\end{equation*}
and
\begin{equation*}\label{eq22}
x_1^{u_1-a-c}x_2^{u_2-d}\equiv \frac{1}{d\bar{\lambda}}(u_2x_1^{u_1-a}x_2^{u_2}+bx_1^{u_1-a}x_2^{u_2+b}) \mod D_{2}^{(\bar{\lambda})}(x_1^{u_1-a}x_2^{u_2}).
\end{equation*}
Hence
$$ x_1^{u_1}x_2^{u_2}\equiv\frac{1}{ad}((cu_2-(u_1-a)d)x_1^{u_1-a}x_2^{u_2}+bcx_1^{u_1-a}x_2^{u_2+b}) \mod \sum_{l=1}^2D_{l}^{(\bar{\lambda})}\bar{R}^{(\bar{\lambda})}.$$
This finishes the proof of claim (1).

Since $u_2<\frac{(d-1)(u_1-a)}{c-1}$ and $u_1\le a$, one has that $u_2<0$.
It follows that
\begin{equation}\label{eq2}
-d<u_2+b< b.
\end{equation}
It follows from $u_2< \frac{(d-1)(u_1-a)}{c-1}$ that
$$u_2+b<\frac{(d-1)(u_1-a)}{c-1}+b.$$
But $-c<u_1-a\le a$ and $\frac{(d-1)(u_1-2a)}{c-1}\le u_2< \frac{(d-1)(u_1-a)}{c-1}+b$. Hence
\begin{equation}\label{eq3}
\frac{(d-1)(u_1-2a)}{c-1}\le  u_2+b<\frac{(d-1)(u_1-a)}{c-1}+b .
\end{equation}
From (\ref{eq000}),(\ref{eq2}) and (\ref{eq3}) we conclude that
$x_1^{u_1-a}x_2^{u_2+b}\in B$. Hence claim (2) is true.

If $x_1^{u_1-a}x_2^{u_2}\notin B$, then $(u_1,u_2)$ satisfies that
$$u_2<\frac{(d-1)(u_1-2a)}{c-1}, \ {\rm or} \ u_1-a\le -c,\ {\rm or}\ u_2\ge \frac{(d-1)(u_1-a)}{c-1}+b.$$
If $u_1-a\le -c$, then
$$\frac{(d-1)(u_1-a)}{c-1}-1\le \frac{-c(d-1)}{c-1}-1< -d.$$
But $u_2>-d$. We have
$$ \frac{(d-1)(u_1-a)}{c-1}< -d+1\le u_2,$$
which contradicts to the assumption $u_2< \frac{(d-1)(u_1-a)}{c-1}$.
Hence $-c<u_1-a< a$.

By the hypothesis $u_2<\frac{(d-1)(u_1-a)}{c-1}$ and $b>0$, we have that
$$u_2< \frac{(d-1)(u_1-a)}{c-1}+b.$$
Hence $u_2<\frac{(d-1)(u_1-2a)}{c-1}$. This implies that $(u_1-a,u_2)$ satisfies
(\ref{eq000}).

If $x_1^{u_1-a}x_2^{u_2+b}\notin B$, then $(u_1,u_2)$ satisfies that
$$u_2+b<\frac{(d-1)(u_1-2a)}{c-1}, \ {\rm or} \ u_1-a\le -c,\ {\rm or}\ u_2+b\ge \frac{(d-1)(u_1-a)}{c-1}+b.$$
We have proved that $-c<u_1-a< a$. Note that $-d<u_2\le 0$. Then $-d<u_2+b\le b$. Clearly, $u_2+b< \frac{d-1}{c-1}(u_1-a)+b$.
Hence $(u_1-a,u_2+b)$ satisfies
(\ref{eq000}). This finishes the proof of claim (3).

Now we prove that Lemma \ref{lem1} holds for those $(v_1,v_2)$ satisfying (\ref{eq000}).
By claim (1), one has that $x_1^{v_1}x_2^{v_2}$ is a linear combination of $x_1^{v_1-a}x_2^{v_2}$
and $x_1^{v_1-a}x_2^{v_2+b}$ over $\mathbb{F}_q$. If $x_1^{v_1-a}x_2^{v_2}\in B$, it then follows from claim (2)
that $x_1^{v_1-a}x_2^{v_2+b}\in B$. Thus Lemma \ref{lem1} holds and the process stop.
If $x_1^{v_1-a}x_2^{v_2}\notin B$ and $x_1^{v_1-a}x_2^{v_2+b}\notin B$, then by claim (3) one has that $(v_1-a,v_2)$ and
$(v_1-a,v_2+b)$ satisfies (\ref{eq000}).
It follows from claim (1) that $x_1^{v_1-a}x_2^{v_2}$ is a linear combination of $x_1^{v_1-2a}x_2^{v_2}$ and $x_1^{v_1-2a}x_2^{v_2+b}$,
and $x_1^{v_1-a}x_2^{v_2+b}$ is a linear combination of $x_1^{v_1-2a}x_2^{v_2+b}$ and $x_1^{v_1-2a}x_2^{v_2+2b}$.
Then we check whether $x_1^{v_1-2a}x_2^{v_2}$, $x_1^{v_1-2a}x_2^{v_2+b}$, $x_1^{v_1-2a}x_2^{v_2+b}$ and $x_1^{v_1-2a}x_2^{v_2+2b}$
are in $B$ or not. If they are, then $x_1^{v_1-a}x_2^{v_2}$ and $x_1^{v_1-a}x_2^{v_2+b}$ are linear combinations of elements in $B$.
So the process stops and Lemma \ref{lem1} is proved. If there are some elements not in $B$, then
we repeat the process. Note that the exponents of $x_1$ decrease and exponents of $x_2$ increase. Hence the process will stop at most $[\frac{v_1}{a}-\frac{(c-1)v_2}{(d-1)a}]$ steps. This proves that Lemma \ref{lem1} is true in this case.
If $x_1^{v_1-a}x_2^{v_2}\notin B$ and $x_1^{v_1-a}x_2^{v_2+b}\in B$, then Lemma \ref{lem1} can be proved similarly.

This finishes the proof of Lemma \ref{lem1}.
\end{proof}

\begin{lem}\label{thm3}
Fix $\bar{\lambda}\in \mathbb{F}_q^{*}$.
For any $m,n\in \mathbb{Z}$, monomial $x_1^mx_2^n$ is a linear combination of elements in $B$ over $\mathbb{F}_q$ modulo $\sum_{l=1}^2D_{l}^{(\bar{\lambda})}\bar{R}^{(\bar{\lambda})}$.
\end{lem}
\begin{proof}
We first prove Lemma \ref{thm3} is true for $m> -c$ and $-d<n\le b$. It follows from Lemma \ref{lem1} that $x_1^mx_2^n$ is a linear combination of
elements in $B$ when $-c<m\le 0$.
Hence we let $m>0$.
Write
$$m=m_1a+a_1 \ {\rm with}\ 0<a_1\le a.$$
It is obvious that $m_1\ge 0$. By Lemma \ref{lem1}, we conclude that Lemma \ref{thm3} is true when $m_1=0$.
Assume Lemma \ref{thm3} holds for all integers less than $m_1$.
In what follows we prove Lemma \ref{thm3} is true for the $m_1+1$ case.
Now suppose $m=(m_1+1)a+a_1$ with $0<a_1\le a$.
Let $D_{1}^{(\bar{\lambda})}$ act on $x_1^{m_1a+a_1}x_2^n$. Then one has
$$ D_{1}^{(\bar{\lambda})}(x_1^{m_1a+a_1}x_2^n)=(m_1a+a_1)x_1^{m_1a+a_1}x_2^n+a_1x_1^{(m_1+1)a+a_1}x_2^n-c\bar{\lambda} x_1^{m_1a+a_1-c}x_2^{n-d}.$$
Hence
$$ x_1^{(m_1+1)a+a_1}x_2^n\equiv \frac{1}{a_1}\Big(c\bar{\lambda} x_1^{m_1a+a_1-c}x_2^{n-d}-(m_1a+a_1)x_1^{m_1a+a_1}x_2^n\Big) \mod D_{1}^{(\bar{\lambda})}(x_1^{m_1a+a_1}x_2^n).$$

By the assumption, it is enough to show that $x_1^{m_1a+a_1-c}x_2^{n-d}$ is a linear combination of elements in $B$. In what follows,
we divide the proof into three cases.

{\sc Case 1.} $n\ge 1$ and $a_1>c$. Then $0<a_1-c<a$ and $-d<n-d<b$. By the hypothesis, monomial $x_1^{m_1a+a_1-c}x_2^{n-d}$ is a linear combination of
elements in $B$ modulo $\sum_{l=1}^2D_{l}^{(\bar{\lambda})}\bar{R}^{(\bar{\lambda})}$.

{\sc Case 2.} $n\ge 1$ and $a_1\le c$.  Clearly, $m_1a+a_1-c>-c$. If there exists a nonnegative integer $m_1'$ such that
$m_1a+a_1-c=m_1'a+a_1'$ with $0<a_1'\le a$, then $m_1'\le m_1$.
By the assumption, we conclude that $x_1^{m_1a+a_1-c}x_2^{n-d}$ is a linear combination of elements in $B$. If there is no such nonnegative integer,
then $-c<m_1a+a_1-c<a$. From Lemma \ref{lem1} we obtain that
$x_1^{m_1a+a_1-c}x_2^{n-d}$ is a linear combination of elements in $B$.

{\sc Case 3.} $n\le 0$. Using $D_{2}^{(\bar{\lambda})}$ to act on $x_1^{m_1a+a_1}x_2^{n}$, one obtains that
$$D_{2}^{(\bar{\lambda})}(x_1^{m_1a+a_1}x_2^{n})=nx_1^{m_1a+a_1}x_2^{n}+bx_1^{m_1a+a_1}x_2^{n+b}-d\bar{\lambda} x_1^{m_1a+a_1-c}x_2^{n-d}.$$
That is,
$$ x_1^{m_1a+a_1-c}x_2^{n-d}\equiv \frac{1}{d\bar{\lambda}}(nx_1^{m_1a+a_1}x_2^{n}+bx_1^{m_1a+a_1}x_2^{n+b}) \mod D_{2}^{(\bar{\lambda})}(x_1^{m_1a+a_1}x_2^{n}).$$
Note that $-d <n+b\le b$. It follows from the assumption that $x_1^{m_1a+a_1}x_2^{n+b}$ is a linear combination of
elements in $B$ over $\mathbb{F}_q$. So is $x_1^{m_1a+a_1-c}x_2^{n-d}$.

Hence Lemma \ref{thm3} holds for $m> -c$ and $-d<n\le b$.
Similar arguments can be used to prove Lemma \ref{thm3} holds for $m> -c$ and $n> -d$.
It remains to show the truth of Lemma \ref{thm3} for $m\le -c$ or $n\le -d$.
Without loss of generality, we just prove the case $m\le -c$. Then we use $D_{1}^{(\bar{\lambda})}$ to act on $ x_1^{m+c}x_2^{n+d}$. It follows that
$$D_{1}^{(\bar{\lambda})}(x_1^{m+c}x_2^{n+d})=(m+c)x_1^{m+c}x_2^{n+d}+ax_1^{m+c+a}x_2^{n+d}-c\bar{\lambda} x_1^{m}x_2^{n}.$$
That is,
\begin{equation}\label{eq23}
x_1^{m}x_2^{n}\equiv \frac{1}{c\bar{\lambda}}\big((m+c)x_1^{m+c}x_2^{n+d}+ax_1^{m+c+a}x_2^{n+d}\big)\ \mod D_{1}^{(\bar{\lambda})}(x_1^{m+c}x_2^{n+d}).
\end{equation}
If $m+c$ and $m+c+a$ are greater than $-c$, then $x_1^{m}x_2^{n}$ is a linear combination of elements in $B$.
The process stops and Lemma \ref{thm3} is true.
If $m+c\le -c$ or $m+c+a\le -c$, then we continue to use $D_{1}^{(\bar{\lambda})}$ to act on $x_1^{m+2c}x_2^{n+2d}$ or $x_1^{m+2c+a}x_2^{n+2d}$.
Observe that the exponent of monomials on the right hand side of (\ref{eq23}) is greater than the left hand side one. Hence after finite steps,
there exist monomials $x_1^{u_1}x_2^{u_2}$ with
$u_1> -c$ and $u_2>-d$ expressing $x_1^{m}x_2^{n}$ linearly.
Hence Lemma \ref{thm3} holds for all monomials $x_1^{m}x_2^{n}$ such that $m\le -c$.

This finishes the proof of Lemma \ref{thm3}.

\end{proof}

{\bf Remark.} The proofs of Lemma \ref{lem1} and \ref{thm3} actually
indicate a stronger result that for any $h\in\bar{R}^{(\bar{\lambda})}$, there are elements $h_1$ and $h_2$
with $h_1, h_2\in \bar{R}^{(\bar{\lambda})}$ such that  $h-\sum_{l=1}^2D_{l}^{(\bar{\lambda})}h_l$
is a linear combination of elements in $B$.

By Lemma \ref{thm3} and Theorem 3.2 of \cite{HS17}, one has the following generalization of Theorem \ref{thm10}.
\begin{thm} \label{thm6}
For $q=p^{\tilde{a}}$ with $\tilde{a}\in \mathbb{Z}_{\ge 0}$,
let $$\bar{B}_q:=\{\Lambda^{qm(v)}x_1^{v_1}x_2^{v_2}: x_1^{v_1}x_2^{v_2}\in B\}.$$
Then set $\bar{B}_q$ forms a basis for $H^2(\Omega^{\bullet}(\mathcal{C}_{0,q}, \nabla(D^{(\Lambda^q)}) ))$.
\end{thm}
Clearly, Theorem \ref{thm10} is the special case $q=1$.

\subsection{Newton polygon and Hodge polygon}
Let $f=\sum a_jx^{V_j}$ be a nondegenerate Laurent polynomial with $n$ variables over $\mathbb{F}_q$.
Let ${\rm Supp}(f):=\{V_j: a_j \neq 0\}$.
Assume $\Delta=\Delta(f)$ is the convex closure of ${\rm Supp} (f)$ and the origin.
Let $S_{\Delta}:=\mathbb{F}_q[x^{Cone(\Delta)\cap \mathbb{Z}^n}]$.
Let
$$H_{\Delta}(i):=\dim_{\mathbb{F}_q}(S_{\Delta}/\sum_{l=1}^n D_{l,f}S_{\Delta})_i,$$
the dimension of the graded degree $i$ part of $S_{\Delta}/\sum_{l=1}^n D_{l,f}S_{\Delta}$, where $D_{l,f}=x_l\frac{\partial}{\partial x_l}+x_l\frac{\partial f}{\partial x_l}$.
Let $w$ be the weight function on $Cone(\Delta)$ and $\tilde{D}$ be the least positive integer such that $w(u)\in (1/\tilde{D})\mathbb{Z}_{\ge 0}$ for all $u\in Cone(\Delta)\cap \mathbb{Z}^n$.
The {\it Hodge polygon} $HP(\Delta)$ of $\Delta$
is the lower convex polygon in $\mathbb{R}^2$ with vertices
$$\Big(\sum_{k=0}^{m}H_{\Delta}(k),\sum_{k=0}^{m}\frac{k}{\tilde{D}}H_{\Delta}(k) \Big), m=0,1,...,n\tilde{D}.$$
It follows from \cite{AS89} that
$$NP(f)\ge HP(\Delta).$$
The polynomial $f$ is called {\it ordinary} if $NP(f)= HP(\Delta)$.

\begin{thm}\label{thm8}(Facial decomposition theorem, \cite{WD1}) Let $f$ be a nondegenerate Laurent polynomial with $n$ variables over $\mathbb{F}_q$.
Assume $\Delta=\Delta(f)$ is $n$-dimensional and $\Delta_1,\cdots,\Delta_h$ are all the codimension 1 faces of $\Delta$
which do not contain the origin. Let $f^{\Delta_i}$ denote the restriction of $f$ to $\Delta_i$. Then $f$ is ordinary if and only if
$f^{\Delta_i}$ is ordinary for $1\le i\le h$.
\end{thm}

In what follows, we introduce some criteria to determine the nondegenerate and ordinary property.

A Laurent polynomial $f\in \mathbb{F}_q[x_1^{\pm},\cdots,x_n^{\pm}]$ is called {\it diagonal} if $f$ has exactly $n$ non-constant terms and $\Delta(f)$ is $n$-dimensional.
Let $f(x)=\sum_{j=1}^na_jx^{V_j}$ with $a_j\in \mathbb{F}_q^{*}$.
The square matrix of $\Delta$ is defined to be
$$\mathbf{M}(\Delta)=(V_1,\cdots,V_n),$$
where each $V_j$ is written as a column vector.
If $f$ is diagonal, then $\det\mathbf{M}(\Delta)\neq0$.

\begin{prop}\label{prop1} Suppose $f\in \mathbb{F}_q[x_1^{\pm},\cdots,x_n^{\pm}]$ is diagonal with $\Delta(f)$. Then $f$ is nondegenerate if and only if
$\gcd(p,\det \mathbf{M}(\Delta))=1$.
\end{prop}
Let $S(\Delta)$ be the solution set of the following linear system
$$\mathbf{M}(\Delta)\cdot(r_1,r_2,...,r_n)^t \equiv 0 \pmod 1, r_i\in \mathbb{Q}\cap [0,1).$$
Then $S(\Delta)$ is an abelian group and its order is given by $|\det \mathbf{M}(\Delta)|$.
By the fundamental structure of finite abelian group, $S(\Delta)$ can be decomposed into a direct product of invariant factors
$$S(\Delta)=\bigoplus_{i=1}^n \mathbb{Z}/d_i\mathbb{Z},$$
where $d_i|d_{i+1}$ for $i=1,2,...,n-1$. Then Wan proved the following ordinary criterion.
\begin{prop}\label{prop2}\cite{WD2} Suppose $f\in \mathbb{F}_q[x_1^{\pm},\cdots,x_n^{\pm}]$ is a nondegenerate diagonal Laurent polynomial with
$\Delta(f)$. Let $d_n$ be the largest invariant factor of $S(\Delta)$. If $p\equiv 1 \mod d_n$, then $f$ is ordinary at $p$.
\end{prop}

We come back to our case $\bar{F}(\Lambda,x)=x_1^a+x_2^b+\frac{\Lambda}{x_1^cx_2^d}$.
By (\ref{eq27}), the Newton polygon of $L(\bar{F}(\bar{\lambda},x), T)^{-1}$ and its lower bound are independent of the choice of $\bar{\lambda}$.
Without loss of generality, we fix $\bar{\lambda}\in \mathbb{F}_q^{*}$.
Now we use Wan's theorems to prove Theorem \ref{thm9}.
\begin{proof}
There are three codimension 1 faces of $\Delta$
which do not contain the origin, denoted by $\Delta_1,\Delta_2,\Delta_3$.
Then let $\bar{F}^{\Delta_1}=x_1^a+x_2^b$, $\bar{F}^{\Delta_2}=x_2^b+\frac{\bar{\lambda}}{x_1^cx_2^d}$
and $\bar{F}^{\Delta_3}=x_1^a+\frac{\bar{\lambda}}{x_1^cx_2^d}$.
We compute $|\det \mathbf{M}(\Delta_1)|=ab$,$|\det \mathbf{M}(\Delta_2)|=bc$ and $|\det \mathbf{M}(\Delta_3)|=ad$.
Note that $p\nmid abcd$. It then follows from Proposition \ref{prop1} that $\bar{F}^{\Delta_1}$, $\bar{F}^{\Delta_2}$ and $\bar{F}^{\Delta_3}$
are nondegenerate.

Recall that $\gcd(a,b)=\gcd(b,c)=1$. If $\gcd(a,d)=1$, $p\equiv 1 \mod ab[c,d]$, then by Proposition \ref{prop2} one has $NP(\bar{F}^{\Delta_i})=HP(\Delta_i)$ for $i=1,2,3$.
It follows from Theorem \ref{thm8} that $NP(\bar{F})=HP(\Delta(\bar{F}))$ at $p$.
This finishes the proof of Theorem \ref{thm9}.
\end{proof}

{\bf Remark:} Generally, the Hodge polygon $HP(\Delta(\bar{F}))$
can be obtained by computing the weights of monomials in set $B$.
The following is how to compute the Hodge polygon.
Let $u=(u_1, u_2)$ such that $-c< u_1\le a$ and $-d<u_2\le b$.
If $u$ satisfies that $0<u_1\le a$ and $0\le u_2\le b$, then let
$$w(u)=\frac{u_1}{a}+\frac{u_2}{b}.$$
If $u$ such that $\frac{d-1}{c-1}(u_1-a)\le u_2\le \frac{du_1}{c}$ and $u_2<0$, then we let
$$w(u)=\frac{u_1}{a}-\frac{(a+c)u_2}{ad}.$$
If $\frac{du_1}{c}<u_2<\frac{d-1}{c-1}u_1+b$ and $u_1\le 0$, then let
$$w(u)=\frac{u_2}{b}-\frac{(b+d)u_1}{bc}.$$
Counting the number of $u$ whose weights are the same, then we denote $H_{\Delta}(k)=card\big\{u: w(u)=\frac{k}{ab[c,d]}\big\}$.
The Hodge polygon is the lower convex polygon in $\mathbb{R}^2$ with vertices
$$\Big(\sum_{k=0}^{m}H_{\Delta}(k),\sum_{k=0}^{m}\frac{k}{ab[c,d]}H_{\Delta}(k) \Big), m=0,1,...,2ab[c,d].$$

Hence under the condition of Theorem \ref{thm9}, we can compute the Newton polygon of the $L$-function $L(\bar{F}(\bar{\lambda},x),T)^{-1}$.
Especially, we have the following.
\begin{cor}Let $c=d=1$ and $p\equiv 1 \mod ab$. Then the slope sequence of $L(\bar{F}(\bar{\lambda},x),T)^{-1}$ is $\{\frac{ai+bj}{ab}\}_{i=0,...,b,j=0,...,a}$.
\end{cor}
\begin{proof}
It follows from Theorem \ref{thm9} that $NP(\bar{F})=HP(\Delta(\bar{F}))$.
The weight of $x_1^ix_2^j$ is $(aj+bi)/ab$.
Since $\gcd(a,b)=1$, those $ab+a+b$ points $\{aj+bi\}_{i=0,...,a,j=0,...,b}-\{aj+bi\}_{i=0,j=b}$ are different with each other.
 Thus the slope sequence of $L(f,T)^{-1}$ is $\{\frac{ai+bj}{ab}\}_{i=0,...,b,j=0,...,a}$.
 \end{proof}

\section{Differential equation}

\subsection{Introduction to the GKZ system}

The theory of GKZ system was established by Gelfand, Kapranov and
Zeleviasky as a generalisation of hypergeometric differential equations \cite{GKZ90, GKZ91, GZK89, GZK93}.
In this part, we just give a short introduction to the GKZ system as far as we need it.

Let $\mathcal{A}=\{a_1,\cdots,a_m\}\subset \mathbb{Z}^n$ be a collection of $m>n$ points lying on an integral affine hyperplane.
Denote $a_i$ by $a_i:=(a_{i1},\cdots, a_{in})^t$.
Let $\alpha=(\alpha_1,\cdots,\alpha_n)\in \mathbb{C}^n$ be an arbitrary vector.
Then $$\mathbb{L}:=\{(l_1,\cdots,l_m)\in \mathbb{Z}^m: l_1a_1+\cdots+l_ma_m=0,a_i\in \mathcal{A}\}$$
denotes the lattice of linear relations among $\mathcal{A}$ .

We define the {\it GKZ system} (sometimes also called $\mathcal{A}$ system) for $\mathcal{A}$ and
$\alpha$ to be a system of differential equations for functions $\Phi$
of $m$ variables $v_1,\cdots,v_m$ given by
\begin{equation}\label{eq5}
\prod_{l_i>0}\Big( \frac{\partial }{\partial v_i}\Big)^{l_i}\Phi= \prod_{l_i<0}\Big( \frac{\partial }{\partial v_i}\Big)^{-l_i}\Phi\ {\rm for\ every}\ l\in \mathbb{ L}
\end{equation}
and
\begin{equation}\label{eq6}
\sum_{i=1}^ma_{ij}v_i \frac{\partial \Phi}{\partial v_i}=\alpha _j\Phi\ {\rm for\ all}\ j=1,\cdots,n\ {\rm and}\ (a_{i1},\cdots,a_{in})\in \mathcal{A}.
\end{equation}

\subsection{Differential equation}

We start with calculating the GKZ system for the one-parameter family $\bar{F}(\Lambda,x)=x_1^a+x_2^b+\frac{\Lambda}{x_1^cx_2^d}$.
Define a general 3-parameter family
$$f_{\bar{v}}(x)=f_{v_1,v_2,\Lambda}(x):=v_1x_1^a+v_2x_2^b+\frac{\Lambda}{x_1^cx_2^d}$$
with parameters $v_1,v_2\in \mathbb{Z}_q$ and $\Lambda$.
Hence we calculate the GKZ system for $$\mathcal{A}=\{(a,0)^t,(0,b)^t,(-c,-d)^t\}$$
and $\alpha=(0,0)^t$.
Let $A$ be the matrix with column vectors in $\mathcal{A}$. Then $A$ is a $2\times 3$ matrix and $\mathbb{L}$ is 1-dimensional.
We solve
\begin{equation*}
{\left( \begin{array}{ccc}
a& 0& -c\\
0& b& -d
\end{array}
\right )}
{\left( \begin{array}{c}
q_1\\
q_2\\
q_3
\end{array}
\right )}= {\left( \begin{array}{c}
0\\
0
\end{array}
\right )}
\end{equation*}
to obtain that
$$\mathbb{L}=\langle(q_1,q_2,q_3)\rangle=\langle(bc,ad,ab)\rangle.$$
Replace operators $\frac{\partial }{\partial \Lambda}$, $\frac{\partial }{\partial v_1}$ and $\frac{\partial }{\partial v_2}$ in
equation (\ref{eq5}) by
$D'_{\Lambda} $, $D'_{1}$ and $D'_{2}$, respectively, where
$$D'_{\Lambda}=\frac{\partial}{\partial \Lambda}+\frac{1}{x_1^cx_2^d},\ D'_{1}=\frac{\partial}{\partial v_1}+x_1^a\ {\rm and}\ D'_{2}=\frac{\partial}{\partial v_2}+x_2^b.$$
Then equation (\ref{eq5}) for lattice $\mathbb{L}$ is
\begin{equation}\label{eq7}
  \Big(D'_{1} \Big)^{bc} \Big(D'_{2} \Big)^{ad} \Big(D'_{\Lambda} \Big)^{ab}\Phi =\Phi.
\end{equation}

Recall that $D_{\Lambda}=\Lambda D'_{\Lambda}$.
Let $D_1:=v_1D'_1$ and $D_2:=v_2D'_2$.
Then $D'_{\Lambda}=\Lambda^{-1} D_{\Lambda}$,
$D'_1=v_1^{-1} D_1$ and $D'_2=v_2^{-1}D_2 $.
Hence (\ref{eq7}) can be written as
\begin{equation}\label{eq8}
(v_1^{-1}D_1)^{bc}(v_2^{-1}D_2)^{ad}(\Lambda^{-1}D_{\Lambda})^{ab}\Phi=\Phi.
\end{equation}
Note that for $j\in \mathbb{Z}$,
$$D_{\Lambda}\cdot\Lambda^j=\Lambda^jD_{\Lambda}+j\Lambda^j$$
 and
$$D_i\cdot v_i^j=v_i^jD_i+jv_i^j, i=1,2.$$
Changing the order of $D_1$ and $v_1^{-1}$, $D_2$ and $v_2^{-1}$, $D_{\Lambda}$ and $\Lambda$, then (\ref{eq8}) is equal to
\begin{equation}\label{eq9}
v_1^{-bc} (D_1-(bc-1))\cdots D_1 v_2^{-ad} (D_2-(ad-1))\cdots D_2 \Lambda^{-ab} (D_{\Lambda}-(ab-1))\cdots D_{\Lambda}\Phi=\Phi.
\end{equation}

The second equation (\ref{eq6}) of the GKZ system given by $\alpha=(0,0)^t$ is as follows
\begin{equation}\label{eq10}
{\left( \begin{array}{ccc}
a& 0& -c\\
0& b& -d
\end{array}
\right )}
{\left( \begin{array}{c}
D_1\\
D_2\\
D_{\Lambda}
\end{array}
\right )}\Phi= {\left( \begin{array}{c}
0\\
0
\end{array}
\right ).}
\end{equation}
Hence
\begin{align*}
\left\{
  \begin{array}{ll}
  aD_1&=cD_{\Lambda}\\
  bD_2&=dD_{\Lambda}.
  \end{array}
\right.
\end{align*}
Replacing $D_1$ and $D_2$ in (\ref{eq9}) by $\frac{cD_{\Lambda}}{a}$ and $\frac{dD_{\Lambda}}{b}$, respectively, and letting $v_1=v_2=1$, one gets that
\begin{equation}\label{eq24}
 \Big(\frac{cD_{\Lambda}}{a}-(bc-1) \Big)\cdots \frac{cD_{\Lambda}}{a}  \Big(\frac{dD_{\Lambda}}{b}-(ad-1) \Big)\cdots \frac{dD_{\Lambda}}{b} \Big(D_{\Lambda}-(ab-1) \Big)\cdots D_{\Lambda}\Phi= \Lambda^{ab} \Phi.
\end{equation}

For $f_{\bar{v}}(x)$, we define
 $$D_{1,\Lambda}:=x_1\frac{\partial}{\partial x_1}+av_1x_1^a-\frac{c\Lambda }{x_1^cx_2^d}$$
and $$D_{2,\Lambda}:=x_2\frac{\partial}{\partial x_2}+bv_2x_2^b-\frac{d\Lambda }{x_1^cx_2^d}.$$
Now we check that $\Phi=\bar{1}=1+\sum_{l=1}^2D_{l,\Lambda}(\mathcal{C}_0)$ is a solution of the GKZ system.
\begin{lem}\label{lem2}
The form $\Phi=\bar{1}$ is a solution of equations (\ref{eq7}) and (\ref{eq10}).
\end{lem}
\begin{proof}
By the definition of $D'_{\Lambda} $, $D'_{1}$ and $D'_{2}$,
one has
$$D'_{\Lambda}(1)=\frac{1}{x_1^cx_2^d},\cdots, (D'_{\Lambda})^{ab}(1)=\big(\frac{1}{x_1^cx_2^d}\big)^{ab}, $$
$$(D_1')^{bc}\big(\frac{1}{x_1^{abc}x_2^{abd}}\big)=\frac{1}{x_2^{abd}}, (D'_{2})^{ad}\big(\frac{1}{x_2^{abd}}\big)=1.$$
By the communication of $D'_{\Lambda} $, $D'_{1}$ and $D'_{2}$ with
$D_{l,\Lambda}$ for $l=1,2$, we conclude that $\Phi=\bar{1}$ satisfies equation (\ref{eq7}).

We compute that
$$aD_1(1)=av_1x_1^a,\ bD_2(1)=bv_2x_2^b \ {\rm and}\ D_{\Lambda}(1)=\frac{\Lambda}{x_1^cx_2^d}.$$
Hence
\begin{align*}
\left\{
  \begin{array}{ll}
  aD_1(1)\equiv cD_{\Lambda}(1) &\mod D_{1,\Lambda}(1), \\
  bD_2(1)\equiv dD_{\Lambda}(1) &\mod D_{2,\Lambda}(1).
  \end{array}
\right.
\end{align*}
That is, $\Phi=\bar{1}$ is a solution of equation (\ref{eq10}).
This finishes the proof of Lemma \ref{lem2}.
\end{proof}
Let $v_1=v_2=1$. Thus $\Phi=\bar{1}$ is a solution of (\ref{eq24}).
To show Theorem \ref{thm11}, it is enough to prove the following result.
\begin{lem}\label{lem6}If $D_{\Lambda}^n(1)$ is a linear combination of $D_{\Lambda}^{n-1}(1),\cdots,D_{\Lambda}(1),1$ over $\mathcal{O}_0$ modulo
$\sum_{l=1}^2D_{l,\Lambda}(\mathcal{C}_0)$, then
$n\ge ab+bc+ad$.
\end{lem}
\begin{proof}
For $i=1,\cdots,n$, we have
$$D_{\Lambda}^i(1)=\frac{\Lambda^i}{x_1^{ic}x_2^{id}}+{\rm lower\ order\ terms \ of}\ \frac{\Lambda}{x_1^{c}x_2^{d}}.$$
Monomial $\frac{\Lambda^n}{x_1^{nc}x_2^{nd}}$ can be reduced as follows
\begin{align*}
D_{1,\Lambda}\Big( \frac{\Lambda^{n-1}}{x_1^{(n-1)c}x_2^{(n-1)d}}\Big)=&\frac{m_1\Lambda^{n-1}}{x_1^{(n-1)c}x_2^{(n-1)d}}+\frac{ax_1^a\Lambda^{n-1}}{x_1^{(n-1)c}x_2^{(n-1)d}}-\frac{\Lambda^{n}}{x_1^{nc}x_2^{nd}}\\
D_{1,\Lambda}\Big( \frac{x_1^a\Lambda^{n-2}}{x_1^{(n-2)c}x_2^{(n-2)d}}\Big)=&\frac{m_2x_1^a\Lambda^{n-2}}{x_1^{(n-2)c}x_2^{(n-2)d}}+
\frac{ax_1^{2a}\Lambda^{n-2}}{x_1^{(n-2)c}x_2^{(n-2)d}}-\frac{x_1^a\Lambda^{n-1}}{x_1^{(n-1)c}x_2^{(n-1)d}}\\
\cdots\cdots\\
D_{1,\Lambda}\Big( \frac{x_1^{(s-1)a}\Lambda^{n-s}}{x_1^{(n-s)c}x_2^{(n-s)d}}\Big)=&\frac{m_sx_1^{(s-1)a}\Lambda^{n-s}}{x_1^{(n-s)c}x_2^{(n-s)d}}+
\frac{ax_1^{sa}\Lambda^{n-s}}{x_1^{(n-s)c}x_2^{(n-s)d}}-\frac{x_1^{(s-1)a}\Lambda^{n-s+1}}{x_1^{(n-s+1)c}x_2^{(n-s+1)d}}\\
D_{2,\Lambda}\Big( \frac{x_1^{sa}\Lambda^{n-s-1}}{x_1^{(n-s-1)c}x_2^{(n-s-1)d}}\Big)=& \frac{m_{s+1}x_1^{sa}\Lambda^{n-s-1}}{x_1^{(n-s-1)c}x_2^{(n-s-1)d}}+
\frac{bx_1^{sa}x_2^{b}\Lambda^{n-s-1}}{x_1^{(n-s-1)c}x_2^{(n-s-1)d}}-\frac{x_1^{sa}\Lambda^{n-s}}{x_1^{(n-s)c}x_2^{(n-s)d}}\\
\cdots\cdots\\
D_{2,\Lambda}\Big( \frac{x_1^{sa}x_2^{(t-1)b}\Lambda^{n-s-t}}{x_1^{(n-s-t)c}x_2^{(n-s-t)d}}\Big)=&\frac{m_{s+t}x_1^{sa}x_2^{(t-1)b}\Lambda^{n-s-t}}{x_1^{(n-s-t)c}x_2^{(n-s-t)d}}+
\frac{bx_1^{sa}x_2^{tb}\Lambda^{n-s-t}}{x_1^{(n-s-t)c}x_2^{(n-s-t)d}}\\
&-\frac{x_1^{sa}x_2^{(t-1)b}\Lambda^{n-s-t+1}}{x_1^{(n-s-t+1)c}x_2^{(n-s-t+1)d}}.
\end{align*}
Adding above equations together, one has
$$\frac{\Lambda^{n}}{x_1^{nc}x_2^{nd}}\equiv \frac{m_{s+t}'x_1^{sa}x_2^{tb}\Lambda^{n-s-t}}{x_1^{(n-s-t)c}x_2^{(n-s-t)d}}+{\rm sums \ of}\ \frac{x_1^{as'}x_2^{bt'}\Lambda^{n-s'-t'-1}}{x_1^{(n-s'-t'-1)c}x_2^{(n-s'-t'-1)d}} \mod \sum_{l=1}^2D_{l,\Lambda}(\mathcal{C}_0),$$
where $m_{s+t}'$ is a constant, $(s',t')$ such that $s'\le s$ and $t'\le t-1$.
It can be checked that $\frac{x_1^{as'}x_2^{bt'}\Lambda^{n-s'-t'-1}}{x_1^{(n-s'-t'-1)c}x_2^{(n-s'-t'-1)d}}$
can be linearly expressed by $D_{\Lambda}^{n-1}(1),\cdots,D_{\Lambda}(1),1$.
It follows that if $D_{\Lambda}^n(1)$ is a linear combination of $D_{\Lambda}^{n-1}(1),\cdots,D_{\Lambda}(1),1$, then
$\frac{x_1^{sa}x_2^{tb}\Lambda^{n-s-t}}{x_1^{(n-s-t)c}x_2^{(n-s-t)d}}$ should equal to $\frac{\Lambda^j}{x_1^{cj}x_2^{dj}}$ for some integer $0\le j\le n-1$. Hence
$$(n-s-t)c=cj+sa, (n-s-t)d=dj+tb.$$
Recall that $\gcd(a,b)=\gcd(a,c)=\gcd(d,b)=1$.
Then $n=bc+ad+ab+j$.
Thus $n\ge ab+bc+ad$. This finishes the proof of Lemma \ref{lem6}.
\end{proof}
Hence Theorem \ref{thm11} holds.
It follows that set $\{\bar{1},D_{\Lambda}(\bar{1}),\cdots, D_{\Lambda}^{N-1}(\bar{1})\}$ consists a basis of
$H^2(\Omega^{\bullet}(\mathcal{C}_0,\nabla(D^{(\Lambda)})))$. The connection map is given as follows.
\begin{cor}\label{thm4}
The action of $D_{\Lambda}$ on the basis $\{\bar{1},D_{\Lambda}(\bar{1}),\cdots, D_{\Lambda}^{N-1}(\bar{1})\}$ is given by
\begin{equation*}
D_{\Lambda}\big(\bar{1},D_{\Lambda}(\bar{1}),\cdots,D_{\Lambda}^{N-1}(\bar{1})\big)^t=G^t(\bar{1},D_{\Lambda}(\bar{1}),\cdots,D_{\Lambda}^{N-1}(\bar{1}))^t,
\end{equation*}
where \begin{equation*}
G^t:=
{\left( \begin{array}{cccc}
0& 1& \cdots&0\\
0& 0& 1 \cdots&0\\
\vdots& \vdots& & \vdots\\
a_0(\Lambda)& a_1(\Lambda)&\cdots &a_{N-1}(\Lambda)
\end{array}
\right )}
\end{equation*} and $a_0(\Lambda),a_1(\Lambda),\cdots, a_{N-1}(\Lambda)$ is determined by (\ref{eq01}).
\end{cor}

\subsection{Frobenius structure}

In the theory of Dwork, Frobenius structure on the differential equation arises in the dual theory, so we introduce the dual theory first.

Let $\mathfrak{C}_0$ be the free $\mathcal{O}_0$-module with basis $\{\pi^{w(v)}\Lambda^{m(v)}x^{v}\}_{v\in M(\bar{f},\mu)}$.
Let $\mathfrak{C}_{0}^{*}$ be the set of all formal power series with coefficients in $\mathcal{O}_0$ over monomials $\{\pi^{-w(v)}\Lambda^{-m(v)}x^{-v}\}_{v\in M(\bar{f},\mu)}$, that is
$$\mathfrak{C}_{0}^{*}:=\Big\{ \sum_{v\in M(\bar{f},\mu)} \xi^{*}(v)\pi^{-w(v)}\Lambda^{-m(v)}x^{-v}: \xi^{*}(v)\in \mathcal{O}_0  \Big\}.$$

For monomial $\frac{1}{\pi^{w(v)}\Lambda^{m(v)}x^v}\in \mathfrak{C}_{0}^{*} $ and $\pi^{w(u)}\Lambda^{m(u)}x^u\in \mathfrak{C}_0$,
the product
\begin{align*}
\frac{\pi^{w(u)}\Lambda^{m(u)}x^u}{\pi^{w(v)}\Lambda^{m(v)}x^v}=\frac{1}{\pi^{w(v)-w(u)}\Lambda^{m(v)-m(u)}x^{v-u}}.
\end{align*}
The total weight function $W$ satisfies that $W(m(v)+m(u);v+u)\le W(m(u);u)+ W(m(v);v)$.
Recall that $W(m(v);v)=w(v)$ and $W_{\Lambda}(\Lambda^r)=r(1-l_{\sigma}(\mu))$.
Replacing $v$ by $v-u$, one has that
$$(m(v-u)+m(u)-m(v))(1-l_{\sigma}(\mu))+w(v)\le w(v-u)+w(u).$$
Thus there exists $s\in \mathbb{R}_{\ge 0}$ such that
\begin{align*}
\frac{\pi^{w(u)}\Lambda^{m(u)}x^u}{\pi^{w(v)}\Lambda^{m(v)}x^v}&=
\frac{\pi^s}{\pi^{(m(v)-m(u)-m(v-u))(1-l_{\sigma}(\mu))+w(v-u) }\Lambda^{m(v)-m(u)}x^{v-u} }\\
&=\frac{\pi^s\pi^{W_{\Lambda}(\Lambda^{m(u)+m(v-u)-m(v)})}\Lambda^{m(u)+m(v-u)-m(v) }}{ \pi^{w(v-u)}\Lambda^{m(v-u)}x^{v-u}}.
\end{align*}
It follows from Lemma \ref{lem3} that $m(v)\le m(u)+m(v-u)$.
Thus $\frac{\pi^{w(u)}\Lambda^{m(u)}x^u}{\pi^{w(v)}\Lambda^{m(v)}x^v}\in \mathfrak{C}_{0}^{*}$, which means the product is well defined.
Hence for $\zeta^{*}\in \mathfrak{C}_{0}^{*}$, $\zeta\in \mathfrak{C}_0$, the product $\zeta^{*}\cdot\zeta$ is well defined.
So we define
$$\langle\zeta^{*},\zeta\rangle:={\rm constant\ term\ of}\ \zeta^{*}\cdot\zeta.$$
Since $\pi^{w(v)}\Lambda^{m(v)}x^v\frac{1}{\pi^{w(v)}\Lambda^{m(v)}x^v}=1$, one has that
$$ \langle, \rangle: \mathfrak{C}_{0}^{*}\times \mathfrak{C}_0\rightarrow \mathcal{O}_0.$$

\begin{lem}\label{lem4}
Under the pairing $\langle, \rangle$, $\mathfrak{C}_{0}^{*}$ is dual to $\mathfrak{C}_0$.
\end{lem}
\begin{proof}
For $\zeta^{*}\in \mathfrak{C}_{0}^{*}$, we define $f_{\zeta^{*}}\in Hom(\mathfrak{C}_0,\mathcal{O}_0)$ by setting
$$f_{\zeta^{*}}(\zeta)=\langle \zeta^{*},\zeta\rangle, \ {\rm for\ any}\ \zeta\in \mathfrak{C}_0.$$
If $\zeta^{*}=\sum_{v\in M(\bar{f},\mu)} \xi^{*}(v)\pi^{-w(v)}\Lambda^{-m(v)}x^{-v}\in \mathfrak{C}_{0}^{*}$ such that $f_{\zeta^{*}}=0$, that is, $ \langle \zeta^{*},\zeta\rangle=0$ for all $\zeta\in \mathfrak{C}_0$, then $\zeta^{*}=0$ since $\pi^{w(v)}\Lambda^{m(v)}x^v\in \mathfrak{C}_0$.
Thus the mapping $\zeta^{*} \mapsto f_{\zeta^{*}}$ is injective.
For $f\in Hom(\mathfrak{C}_0,\mathcal{O}_0)$, we let
$$\zeta_f^{*}:=\sum_{v\in M(\bar{f},\mu)}\frac{f(\pi^{w(v)}\Lambda^{m(v)}x^v)}{\pi^{w(v)}\Lambda^{m(v)}x^v}\in \mathfrak{C}_{0}^{*}.$$
Then for all $\zeta\in \mathfrak{C}_0$,
$$\langle \zeta_f^{*}, \zeta\rangle=f(\zeta).$$
Thus the mapping $\zeta^{*} \mapsto f_{\zeta^{*}}$ is also surjective.
This finishes the proof of Lemma \ref{lem4}.

\end{proof}

For $l=1,2$, we define
$$ D_{l,\Lambda}^{*} :=-x_l\frac{\partial }{\partial x_l}+x_l \frac{\partial F^{(0)}(\Lambda,x)}{\partial x_l}.$$
It is also can be written as
\begin{equation}\label{eq16}
D_{l,\Lambda}^{*} =-\exp F^{(0)}(\Lambda,x)\cdot x_l\frac{\partial }{\partial x_l}\cdot \frac{1}{\exp F^{(0)}(\Lambda,x)}.
\end{equation}
If $\zeta^{*}\in \mathfrak{C}_{0}^{*}$ and $\zeta\in \mathfrak{C}_{0}$, then
$$\langle D_{l,\Lambda}^{*}\zeta^{*},\zeta\rangle=\langle \zeta^{*},D_{l,\Lambda}\zeta\rangle.$$

Let $\{e_1,\cdots,e_N\}$ denote $\bar{B}$. Elements $e_1,\cdots,e_N$ are linearly independent over $\mathcal{O}_0$.
Define space $\mathfrak{B}$ by
$$\mathfrak{B}:=\mathcal{O}_0e_1+\mathcal{O}_0e_2+\cdots+\mathcal{O}_0e_N.$$

\begin{lem}\label{lem5}
As $\mathcal{O}_0$-modules, we have
\begin{equation}\label{eq26}
\mathfrak{C}_0=\mathfrak{B}\oplus \sum_{l=1}^2D_{l,\Lambda}\mathfrak{C}_0.
\end{equation}
\end{lem}

\begin{proof}
Clearly, the left side of (\ref{eq26}) contains the right side.
If $\zeta$ is a polynomial in $\mathfrak{C}_0$ and written as $\zeta=\bigoplus_{v\in M(\bar{f},\mu)}\xi(v)\pi^{w(v)}\Lambda^{m(v)}x^v$, then define $$W(\zeta):=\max\{W(m(v);v)|\xi(v)\neq 0\}.$$
Let
$$\mathfrak{C}_0^{(m)}:=\{\zeta\in \mathfrak{C}_0: W(\zeta)\le m\}$$
for any $m\in (1/D)\mathbb{Z}_{>0}$.
To reverse the inclusion, let $\zeta\in \mathfrak{C}_0^{(m)}$.
By similar arguments as in the proof of Lemma \ref{lem1} and \ref{thm3}, one can derive that
\begin{equation}\label{eq29}
\zeta=\frac{c_1(\Lambda)}{a_1(\Lambda)}e_1+\cdots+\frac{c_N(\Lambda)}{a_N(\Lambda)}e_N+D_{1,\Lambda}h_1+D_{2,\Lambda}h_2,
\end{equation}
where $a_i(\Lambda), c_i(\Lambda)\in \mathcal{O}_0$ for $i=1,\cdots, N$, and
$h_1, h_2$ are polynomials whose coefficients are of the form of $ \frac{g_1(\Lambda)}{g_2(\Lambda)}$ with polynomials $g_1(\Lambda)$ and
$g_2(\Lambda)$.
Note that $\mathfrak{C}_0\subseteq \mathcal{C}_0$. Then $\zeta\in \mathcal{C}_0$. It follows from
Theorem \ref{thm5} that
\begin{equation}\label{eq30}
\zeta=b_1(\Lambda)e_1+\cdots+b_N(\Lambda)e_N+D_{1,\Lambda}\tilde{h}_1+D_{2,\Lambda}\tilde{h}_2
\end{equation}
where $b_i(\Lambda)\in \mathcal{O}_0$ for $i=1,\cdots, N$, and
$\tilde{h}_1,\ \tilde{h}_2 \in \mathcal{C}_0$.
Multiply element $C(\Lambda)\in \mathcal{O}_0$ to (\ref{eq29}) and (\ref{eq30}) so that the coefficients of
$C(\Lambda)\zeta$ are in $\mathcal{O}_0$. We consider the injective map $\iota: \mathfrak{C}_0\rightarrow \mathcal{C}_0$.
Clearly, $\iota(\mathfrak{B})\subseteq \sum_{v\in B}\mathcal{O}_0\pi^{m(v)}\Lambda^{m(v)}x^v$ and $\iota(\sum_{l=1}^2D_{l,\Lambda}\mathfrak{C}_0)\subseteq \sum_{l=1}^2D_{l,\Lambda}\mathcal{C}_0$.
Using the directness of Theorem \ref{thm5}, one has that
$\zeta\in \mathfrak{B}\cap \mathfrak{C}_0^{(m)}+\sum_{l=1}^2D_{l,\Lambda}\mathfrak{C}_0^{(m-1)}$.
Hence
$$\mathfrak{C}_0^{(m)}\subseteq \mathfrak{B}\cap \mathfrak{C}_0^{(m)}+\sum_{l=1}^2D_{l,\Lambda}\mathfrak{C}_0^{(m-1)}.$$
It follows that the right side of (\ref{eq26}) contains the left side.
Hence
$$\mathfrak{C}_0= \mathfrak{B}+\sum_{l=1}^2D_{l,\Lambda}\mathfrak{C}_0.$$

The directness follows from Theorem \ref{thm5}.
The proof of Lemma \ref{lem5} is finished.
\end{proof}

Let $$\mathcal{W}_{\Lambda}=\mathfrak{C}_0/ \sum_{l=1}^2D_{l,\Lambda}\mathfrak{C}_0.$$
The dual space of $\mathcal{W}_{\Lambda}$ is the annihilator in $\mathfrak{C}_{0}^{*}$ of $\sum_{l=1}^2D_{l,\Lambda}\mathfrak{C}_{0}$,
which we denote by
$$\mathcal{W}_{\Lambda}^{*}:=\{\zeta^{*}\in \mathfrak{C}_{0}^{*}: D_{l,\Lambda}^{*}(\zeta^{*})=0,\ l=1,2\}.$$

Now we specialize $\Lambda$ an variable to $\lambda$ a parameter taking values in $\Omega$.
We define the space $\mathcal{W}_{\lambda}^{*}$ obtained by specializing the elements of $\mathcal{W}_{\Lambda}^{*}$ at $\Lambda=\lambda$.
The $\mathcal{O}_0$ dual space $\mathcal{W}_{\Lambda}^{*}$ to $\mathcal{W}_{\Lambda}$ is freely generated by $\{\zeta_{i,\Lambda}^{*}\}_{i=0}^{N-1}$, the dual basis to $\{1,D_{\Lambda}(1),\cdots, D_{\Lambda}(1)^{N-1}\}$. Then
 $\mathcal{W}_{\lambda}^{*}$ is a vector space over $\Omega$ with basis $\{\zeta_{i,\lambda}^{*}\}_{i=0}^{N-1}$.

If $\lambda,z\in \Omega$, then we define operator
$$T_{z,\lambda}:=\exp \pi(\lambda-z)/x^u.$$
Operator $T_{z,\lambda}$ is a well defined when $\lambda,z$ are closed enough.

It is useful to view $T_{z,\lambda}$ as
\begin{equation}\label{eq15}
T_{z,\lambda}= \frac{\exp F^{(0)}(\lambda,x) }{\exp F^{(0)}(z,x)}.
\end{equation}
From (\ref{eq16}) and (\ref{eq15}), we have that
\begin{equation}\label{eq20}
D_{l,\lambda}^{*}\circ T_{z,\lambda}=T_{z,\lambda}\circ D_{l,z}^{*}.
\end{equation}
Hence if $\zeta^{*}\in \mathcal{W}_{z}^{*}$, i.e. $D_{l,z}^{*}( \zeta^{*})=0 $,
then $D_{l,\lambda}^{*}\circ T_{z,\lambda}(\zeta^{*})=0$ for each $l=1,2$.
But $\mathcal{W}_{\lambda}^{*}$ is the kernel of the operator $\{D_{l,\lambda}^{*}\}$ with $l=1,2$.
Hence $T_{z,\lambda}(\zeta^{*})\in \mathcal{W}_{\lambda}^{*}$,
so that $T_{z,\lambda}$ is a map from $\mathcal{W}_{z}^{*}$ to $\mathcal{W}_{\lambda}^{*}$. Since
$T_{z,\lambda}$ is multiplication by an exponential, it is invertible.
Hence $T_{z,\lambda}$ is a linear isomorphism from $\mathcal{W}_{z}^{*}$ to $\mathcal{W}_{\lambda}^{*}$.

Now we view $z$ as fixed and $\Lambda$ as a variable in $\Omega$.
We extend scalars by lifting the base field $\Omega$ of $\mathcal{W}_{\Lambda}^{*}$
and $\mathcal{W}_{z}^{*}$ to the field of functions meromorphic at $z$.
We then have the commutative diagram
\begin{equation}\label{eq03}
\xymatrix{
&\mathcal{W}_{z}^{*} \ar[d]^{\Lambda \frac{\partial}{\partial \Lambda}} \ar[r]_{T_{z,\Lambda}} & \mathcal{W}_{\Lambda}^{*} \ar[d]^{\epsilon^{*}}\\
& \mathcal{W}_{z}^{*} \ar[r]_{T_{z,\Lambda}} & \mathcal{W}_{\Lambda}^{*}},
\end{equation}
where we define on $\mathcal{W}_{\Lambda}^{*}$ that
$$\epsilon^{*}:=\Lambda\frac{\partial}{\partial \Lambda}-\frac{\pi \Lambda}{x^u}.$$
It can be checked that
$$\epsilon^{*}=\exp F^{(0)}(\Lambda,x)\circ \Lambda\frac{\partial}{\partial \Lambda}\circ \frac{1}{\exp F^{(0)}(\Lambda,x)}.$$
By the definition, one has
$$ D_{l,\Lambda}^{*}\circ\epsilon^{*}=\epsilon^{*}\circ D_{l,\Lambda}^{*}.$$
Thus $\epsilon^{*}$ does map $ \mathcal{W}_{\Lambda}^{*}$ to $\mathcal{W}_{\Lambda}^{*}$.

Define a map $\epsilon$ on $ \mathfrak{C}_0$ by
$$\epsilon:= \Lambda\frac{\partial}{\partial \Lambda}+\frac{\pi \Lambda}{x^u}=D_{\Lambda}.$$
Since $\epsilon$ communicates with $D_{1,\Lambda}$ and $D_{2,\Lambda}$, one has that $\epsilon$ induces a map
$$\epsilon: \mathfrak{C}_0/\sum_{l=1}^2D_{l,\Lambda} \mathfrak{C}_0\rightarrow \mathfrak{C}_0/\sum_{l=1}^2D_{l,\Lambda} \mathfrak{C}_0.$$
One can check that
\begin{equation}\label{eq17}
\langle \epsilon^{*}\zeta^{*},\zeta\rangle+\langle \zeta^{*},\epsilon\zeta\rangle=\Lambda\frac{\partial}{\partial \Lambda}
\langle \zeta^{*},\zeta\rangle
\end{equation}
for $\zeta^{*}\in \mathfrak{C}_0^{*}$, $\zeta\in \mathfrak{C}_0$.

We are interested in computing the matrix of $\epsilon^{*}$ with
respect to the basis $\{\zeta_{i,\Lambda}^{*}\}_{i=0}^{N-1}$ of
$\mathcal{W}_{\Lambda}^{*}$.
By (\ref{eq17}), we have
$$\epsilon^{*}(\zeta_{0,\Lambda}^{*},\cdots,\zeta_{N-1,\Lambda}^{*} )^t=-G (\zeta_{0,\Lambda}^{*},\cdots,\zeta_{N-1,\Lambda}^{*} )^t,$$
where $G^t$ is given as Corollary \ref{thm4}.
From (\ref{eq03}), we see that $\zeta^{*}=\sum_{i=0}^{N-1}C_i(\Lambda)\zeta_{i,\Lambda}^{*}$ is the image under $T_{z,\Lambda}$
of element independent of $\Lambda$ if and only if $\epsilon^{*} \zeta^{*}=0$.
Then by (\ref{eq17}), we have
\begin{equation}\label{eq18}
\Lambda\frac{\partial}{\partial \Lambda}(C_0(\Lambda),C_1(\Lambda),\cdots,C_{N-1}(\Lambda))=(C_0(\Lambda),C_1(\Lambda),\cdots,C_{N-1}(\Lambda))G.
\end{equation}

For $\zeta^{*}(x) \in \mathcal{W}_{\Lambda^p}^{*}$, we define the map $\alpha_{\Lambda}^{*}$ of $ \mathcal{W}_{\Lambda^p}^{*}$ into $\mathcal{W}_{\Lambda}^{*}$ by
$$\alpha_{\Lambda}^{*}( \zeta^{*}(x))=F^{(0)}(\Lambda,x)\zeta^{*}(x^p).$$
We define $\Psi_p$ by
$$\Psi_p(\sum A(v)x^v)=\sum A(v)x^{pv}.$$
Then
$$\alpha_{\Lambda}^{*}=\exp F^{(0)}(\Lambda,x)\circ \Psi_p \circ \frac{1}{\exp F^{(0)}(\Lambda,x)}.$$
When $\Lambda$ are close to $z$, the following diagram
\begin{equation}\label{eq19}
\xymatrix{
&\mathcal{W}_{z^p}^{*} \ar[d]^{\alpha_z^{*}} \ar[r]_{T_{z^p,\Lambda^p}} & \mathcal{W}_{\Lambda^p}^{*} \ar[d]^{\alpha_{\Lambda}^{*}}\\
& \mathcal{W}_{z}^{*} \ar[r]_{T_{z,\Lambda}} & \mathcal{W}_{\Lambda}^{*}}
\end{equation}
commutes.

Now we give the proof of Theorem \ref{thm7}.
If $C=(C_0(\Lambda),C_1(\Lambda),\cdots,C_{N-1}(\Lambda))$ is a solution of (\ref{eq18}), then
$$\zeta^{*}=\sum_{i=0}^{N-1}C_i(\Lambda) \zeta_{i,\Lambda}^{*}=T_{z,\Lambda}(\zeta_z^{*}), $$
where $\zeta_z^{*}$ is an element of $W_z^{*}$ which is independent of $\Lambda$.
We write $C^{\sigma \varphi}$ for the process of replacing each coefficient by its image under $\sigma$ and replacing $(\Lambda-a)$
by $(\Lambda^p-\sigma a)$.
We apply $\sigma$ to $\zeta_z^{*}$, then
\begin{equation*}
\sum_{i=0}^{N-1}C_i(\Lambda)^{\sigma \varphi} \zeta_{i,\Lambda^p}^{*} =T_{\sigma z,\Lambda^p}(\zeta_z^{*})^{\sigma}.
\end{equation*}
Thus
\begin{equation}\label{eq28}
\alpha_{\Lambda}^{*}\Big(\sum_{i=0}^{N-1}C_i(\Lambda)^{\sigma \varphi} \zeta_{i,\Lambda^p}^{*}\Big) =\alpha_{\Lambda}^{*}\circ T_{\sigma z,\Lambda^p}(\zeta_z^{*})^{\sigma}.
\end{equation}
Note that
$$\alpha_{\Lambda}^{*}\circ T_{\sigma z,\Lambda^p}=T_{z,\Lambda}\circ \alpha_z^{*}\circ T_{\sigma z,z^p}.$$
Hence the left side of (\ref{eq28}) is the image under $T_{z,\Lambda}$ independent of $\Lambda$.
Thus
$$(C_0(\Lambda),C_1(\Lambda),\cdots,C_{N-1}(\Lambda))^{\sigma\varphi}U(\Lambda)$$ is a solution of equation of (\ref{eq18}),
where $U(\Lambda)$ is the matrix of $\alpha_{\Lambda}^{*}$ with respect to the basis
 basis $\{\zeta_{i,\Lambda}^{*}\}_{i=0}^{N-1}$. By the duality, one also has that $U(\Lambda)$ is the matrix of
$\alpha$ with respect to the basis $\{1, D_{\Lambda}(1),\cdots, D_{\Lambda}^{N-1}(1)\}$.
Hence Theorem \ref{thm7} is true.

This shows that (\ref{eq18}) has a strong Frobenius structure in the sense of Dwork.
Dwork \cite{[Dwo742]} has conjectured that the Frobenius structures of differential equations exists quite widely.
Hence our work confirms this conjecture for our family.

\begin{center}
{\sc Acknowledgements}
\end{center}
The authors would like to thank Steven Sperber for drawing their
attention to the problem considered in this paper, and particularly
for many helpful and stimulating discussions
during their visiting at the University of Minnesota. We thank DaQing Wan for his
invaluable help in this work. Liping Yang also thanks Huaiqian Li and Mingfeng Zhao for many helpful discussions.
\bibliographystyle{amsplain}

\begin{thebibliography}{10}

\bibitem{AS89} A. Adolphson and S. Sperber, Exponential sums and Newton
polyhedra: cohomology and estimates, {\it Ann. of Math.} {\bf 130} (1989), 367-406.

\bibitem{CK99} D. A. Cox and S. Katz, {\it Mirror symmtry and algebbraic geometry, Volume 68 of Mathematical Surveys and Monographs},
American Mathematical Society, Providence, RI, 1999.
\bibitem{DKS} C. F. Doran, T. Kelly, A. Salerno, S. Sperber, J. Voight and U. Whitcher, Zeta functions of alternate
mirror Calabi-Yau families, {\it Israel J. Math.} {\bf 228} (2018), 665-705.

\bibitem{[Dwo62]} B. Dwork, On the zeta function of a hypersurface, {\it Publ. Math. I.H.E.S.} {\bf 12}(1962), 5-68.

\bibitem {Dwo69} B. Dwork, $p$-Adic cycles, {\it Inst. Hautes $\acute{E}$tudes Sci. Publ. Math.} {\bf 37} (1969), 27-115.

\bibitem{[Dwo74]} B. Dwork, Bessel functions as $p$-adic functions of the argument, {\it Duke Math. J.} {\bf 41}(1974), 711-738.

\bibitem{[Dwo742]} B. Dwork, $p$-adic differential equations I, {\it Bull. Soc. Math. France} {\bf 39-40}(1974), 27-37.

\bibitem{[FW05]}  L. Fu and D. Wan, $L$-functions for symmetric products of Kloosterman sums, {\it J. Reine Angew. Math.} {\bf 589} (2005), 79-103.

\bibitem{[FW08]} L. Fu and D. Wan, $L$-functions of symmetric products of the Kloosterman sheaf over $\mathbf{Z}$, {\it Math. Ann.} {\bf 342} (2008), 387-404.

\bibitem{[FW082]} L. Fu and D. Wan, Trivial factors for $L$-functions of symmetric products of Kloosterman sheaves, {\it Finite Fields Appl.} {\bf 14} (2008), 549-570.

\bibitem{[FW10]} L. Fu and D. Wan, Functional equations of $L$-functions for symmetric products of the Kloosterman sheaf, {\it Trans. Amer. Math. Soc.} {\bf 362} (2010), 5947-5965.
\bibitem{GKZ90} I. M. Gelfand, M. M. Kapranov and A. V. Zelevinsky, Generalized Euler integerals and A-hypergeometric functions, {\it Adv. math.}
{\bf 84} (1990), 255-271.
\bibitem{GKZ91} I. M. Gelfand, M. M. Kapranov, and A. V. Zelevinsky, Hypergeometric functions, toric varieties and Newton polyhedra, {\it Special functions (Okayama, 1990)}, ICM90 Satell. Conf. Proc., Springer, Tokyo, (1991), 104-121.

\bibitem{GZK89} I. M. Gelfand, A. V. Zelevinsky, and M. M. Kapranov, Hypergeometric functions and toric varieties, {\it Funktsional. Anal. i Prilozhen.} {\bf 23} (1989), 12-26.

\bibitem{GZK93} I. M. Gelfand, A. V. Zelevinsky, and M. M. Kapranov, Correction to \cite{GZK89}, {\it Funktsional. Anal. i Prilozhen.} {\bf 27} (1993), 91.
\bibitem{[Gah11]} S. G$\ddot{\text{a}}$hrs, Picard-Fuchs equations of special one-parameter families of invertible polynomials,
Ph.D. thesis, Gottfried Wilhelm Leibniz Univ. Hannover, arXiv:1109.3462.

\bibitem{[Gah13]} S. G$\ddot{\text{a}}$hrs, Picard-Fuchs equations of
special one-parameter families of invertible polynomials in arithmetic
and geometry of K3 surfaces and Calabi-Yau threefolds, {\it Fields Inst. Commun.} {\bf 67} (2013), 282-310.


\bibitem{HS14} C. D. Haessig and S. Sperber, $L$-functions associated with families of toric exponential sums,  {\it J. Number Theory } {\bf 144} (2014), 422-473.

\bibitem{HS17} C. D. Haessig and S. Sperber, Symmetric power $L$-functions for families of generalized Kloosterman sums, {\it Trans. Amer. Math. Soc.} {\bf 369} (2017), 1459-1493.

\bibitem{NK68} N. Katz, On the Differential Equations Satisfied by Period Matrices, {\it Publ. Math. I.H.E.S.} {\bf 35} (1968), 223-258.

\bibitem{[PM70]} P. Monsky, {\it $p$-Adic analysis and zeta functions},
Lectures in Mathematics, Kyoto University, 4 Kinokuniya Book-Store Co., Ltd., Tokyo, 1970.

\bibitem{Ro} P. Robba, Symmetric powers of the $p$-adic Bessel equation, {\it J. Reine Angew. Math.}
{\bf 366} (1986), 194-220.

\bibitem{Sp1} S. Sperber, $p$-Adic hypergeometric functions and their cohomology, {\it Duke Math. J.}
{\bf 44} (1977), 535-589.

\bibitem{Sp2} S. Sperber, Congruence properties of the hyperKloosterman sum, {\it Compositio Math.}
{\bf 40} (1980), 3-33.
\bibitem{WD1} D. Wan, Newton polygons of zeta functions and $L$-functions, {\it Ann. of Math.} {\bf 137}(1993),
249-293.
\bibitem{WD2} D. Wan, Variation of $p$-adic Newton polygons for $L$-functions of exponential sums, {\it Asian J. Math.} {\bf 8} (2004),
427-471.
\end{thebibliography}

\end{document}